\documentclass[12pt, reqno]{amsart}
\usepackage{amsmath, amsthm, amscd, amsfonts, amssymb, graphicx, color, mathrsfs}
\usepackage[bookmarksnumbered, colorlinks, plainpages]{hyperref}
\usepackage[all]{xy} 
\usepackage{slashed}

 \newtheorem{thm}{Theorem}[section]
 \newtheorem{cor}[thm]{Corollary}
 \newtheorem{lem}[thm]{Lemma}
 
 \theoremstyle{definition}
 \newtheorem{defn}[thm]{Definition}
 \theoremstyle{rem}
 \newtheorem{rem}[thm]{Remark}
 
 \numberwithin{equation}{section}

\newcommand{\half}{\frac{1}{2}}

\newcommand{\ene}{\mathbb{N}}

\newcommand{\ce}{\mathbb{C}}

\newcommand{\bi}{\begin{itemize}}

\newcommand{\ei}{\end{itemize}}
\newcommand{\be}{\begin{enumerate}}
\newcommand{\ee}{\end{enumerate}}
\newcommand{\beq}{\begin{equation}}
\newcommand{\eq}{\end{equation}}

\def\p#1{{\left({#1}\right)}}

\def\Hcal{{\mathcal H}}

\DeclareMathOperator{\Tr}{Tr}

\def\N{{{\mathbb N}}}

\def\SU2{{{\rm SU(2)}}}
\def\lapsu2{{{\mathcal L}_\SU2}}

\begin{document}
\setcounter{page}{1}

\title[Dixmier traces,  residues and determinants on compact Lie groups]{Dixmier traces, Wodzicki residues, and determinants on compact Lie groups: the paradigm of the global quantisation}

\author[D. Cardona]{Duv\'an Cardona}
\address{
  Duv\'an Cardona:
  \endgraf
  Department of Mathematics: Analysis, Logic and Discrete Mathematics
  \endgraf
  Ghent University, Belgium
  \endgraf
  {\it E-mail address} {\rm duvanc306@gmail.com, duvan.cardonasanchez@ugent.be}
  }
  
  \author[J. Delgado]{Julio Delgado}
\address{
  Julio Delgado:
  \endgraf
  Departmento de Matem\'aticas
  \endgraf
  Universidad del Valle
  \endgraf
  Cali-Colombia
    \endgraf
    {\it E-mail address} {\rm delgado.julio@correounivalle.edu.co}
  }

\author[M. Ruzhansky]{Michael Ruzhansky}
\address{
  Michael Ruzhansky:
  \endgraf
  Department of Mathematics: Analysis, Logic and Discrete Mathematics
  \endgraf
  Ghent University, Belgium
  \endgraf
 and
  \endgraf
  School of Mathematical Sciences
  \endgraf
  Queen Mary University of London
  \endgraf
  United Kingdom
  \endgraf
  {\it E-mail address} {\rm michael.ruzhansky@ugent.be, m.ruzhansky@qmul.ac.uk}
  }

\thanks{The authors are supported  by the FWO  Odysseus  1  grant  G.0H94.18N:  Analysis  and  Partial Differential Equations and by the Methusalem programme of the Ghent University Special Research Fund (BOF)
(Grant number 01M01021). Julio Delgado is also supported by Vice. Inv. Universidad del Valle Grant CI71281. Michael Ruzhansky is also supported  by EPSRC grant 
EP/R003025/2.
}


\begin{abstract} By  following  the paradigm of the global quantisation, instead of the analysis under changes of coordinates,
in this work we establish a global analysis for the explicit computation of the  Dixmier trace and the Wodzicki residue of (elliptic and subelliptic) pseudo-differential operators on compact Lie groups.  The regularised determinant for the Dixmier trace is also computed. We obtain  these formulae in terms of the global symbol of the corresponding operators. In particular, our  approach links the Dixmier trace and Wodzicki residue to the representation theory of the group.  Although we start by analysing the case of compact Lie groups, we also compute  the Dixmier trace and its regularised determinant  on arbitrary closed manifolds $M$, for the class of  invariant pseudo-differential operators in terms of their matrix-valued symbols. This analysis includes e.g. the family of positive and elliptic pseudo-differential operators on $M$.
\end{abstract} 

\maketitle

\tableofcontents
\allowdisplaybreaks

\section{Introduction}
\subsection{Outline} In this work we compute the  Wodzicki-Guillemin residue  and the Dixmier trace of  pseudo-differential operators on compact Lie groups in terms of their global symbols, or equivalently, in terms of the representation theory of these groups. We prove  our geometric/spectral formulae in the context of pseudo-differential operators associated to Riemannian  and sub-Riemannian structures on compact Lie groups \cite{Ruz,CR20}. The validity of such  results can be conjectured from the  techniques  developed in \cite{CPN,PS07} where the Wodzicki residue was studied by using a twisted version of  the global-symbolic calculus on the  torus developed by the third author with Turunen \cite{Ruz}. See also \cite{cdc20}. Indeed, the theory in \cite{Ruz}, can be understood as a discretisation of the algebraic formalism by Connes and Moscovici \cite{CM95} in the setting of Lie groups,  making that theory suitable for the global analysis of  residues, determinants and traces.

As it was pointed out by Connes (cf. \cite{Connes94}), in algebraic quantum field theory, in order to write down an action in operator language, a functional that replaces integration is required.  For the Yang-Mills theory, that is the Dixmier trace $\textnormal{\bf Tr}_{\omega}$, which in contrast to the usual trace, vanishes on the ideal of trace class operators, and  is defined in  the ideal $\mathcal{L}^{(1,\infty)}$ of the compact operators $T$ such that the partial sums of  their singular values diverge logarithmically. On the other hand, the Wodzicki (or non-commutative) residue, is the only extension of the Dixmier trace on  the algebra of  classical  pseudo-differential operators, in view of a fundamental result due to Connes \cite{Connes94}.   These two traces, are of relevant interests in several branches of  mathematics such as differential geometry, spectral theory, mathematical physics, and  non-commutative geometry (c.f.  \cite{Connes94}).  A reason for this is that they  allow to study specific topological/geometric invariants, where very specific elliptic operators --such as the Dirac operator-- have been shown to contain very important information relevant for classification problems, see Lesch \cite{Lesch}, Paycha \cite{Paycha}, Sukochev and Usachev \cite{Sukochev}, Scott \cite{Scott}, Grubb and Schrohe \cite{GrubbSchrohe}, Fedosov, Golse, Leichtnam, and Schrohe \cite{FedosovGolseLeichtnam}  and references therein, just to mention a few.

According to the theory  developed by H\"ormander \cite{Hormander1985III}, to any pseudo-differential operator $A$ acting on $C^{\infty}(M),$ with $M$ a  closed manifold $M$  (i.e. compact and without boundary) one can associate a symbol, defined by local coordinate systems, which is a function on the cotangent bundle $T^*M.$ However, some information about the operator (such as its lower terms) can be lost from the geometric point of view.  On the other hand, it was described by the third author and Turunen in \cite{Ruz}, that if $M=G$ is a compact Lie group, one can give a global definition of symbol for the H\"ormander classes, and that these symbols can be realised as matrix-valued functions of the phase space $G\times \widehat{G},$ where $\widehat{G}$ is the unitary dual of $G$, that is the set of  equivalence classes of continuous irreducible unitary representations on the group.

In this work we address the following problems:
(1) to compute the Dixmier trace and the regularised determinant associated with it, for  pseudo-differential operators when $M=G$ is a compact Lie group. The reason of considering the case of a compact Lie group is to exploit the symmetries of the group as in \cite{Ruz}. Also, by following \cite{CR20}, we also consider the problem of computing the Dixmier trace of subelliptic pseudo-differential operators. 
(2) To compute the Wodzicki residue of (elliptic and subelliptic) pseudo-differential operators on a compact Lie group $G$ in terms of their matrix-valued  symbols or equivalently, in terms of the representation theory of the group, and (3) to compute the regularised determinant for the Dixmier trace in the case of invariant pseudo-differential operators on a compact manifold $M$ in terms of their matrix-valued symbols, by following the formalism developed in \cite{dr14a:fsymbsch}, where the global symbols are defined in terms of a global Fourier analysis on  $M$  and  covers the case of elliptic and positive pseudo-differential operators on $M.$

\subsection{Main results}
 For simplicity, we will discuss briefly our main results in the setting of positive (elliptic or subelliptic) pseudo-differential operators. However, we present  the general situation without positivity assumptions in Section \ref{SectioDixmier}.    
 
 By following H\"ormander \cite{Hormander1985III},  a classical  pseudo-differential operator $A$ on a closed manifold $M$ of dimension $n,$ has a symbol $$\sigma^A(x,\xi)\sim \sum_{j=0}^{\infty}\sigma^A_{m-j}(x,\xi),\,\,(x,\xi)\in T^*M,$$ defined by localisations,  and 
the Wodzicki residue, which measures  the locality of the operator, is given by 
\begin{equation}\label{res}
\textnormal{res}\,(A)=\frac{1}{n(2\pi)^n}\int\limits_{M}\int\limits_{| \xi|=1}\sigma^A_{-n}(x,\xi)\, d\xi\,dx.
\end{equation}By a fundamental theorem of Connes \cite{Connes94}, if $A$ has order $-n,$ then $\textnormal{res}\,(A)$ coincides with the Dixmier trace of $A$, that is: 
\begin{equation*}
\boxed{{    \textnormal{ Spectral (global) information }=\textnormal{\bf Tr}_{\omega}(A)=\textnormal{res}\,(A)=  \textnormal{ Geometric (local) information}}}
\end{equation*}
Now, if $M$ has symmetries, for instance if $M$ is the torus, $M=\textnormal{SU}(2),$ or $M=G$ is an arbitrary compact Lie group, the work of the third author and Turunen \cite{Ruz} shown that for describing the  H\"ormander classes in \cite{Hormander1985III}, is enough to consider the representation theory of $G.$ Then, a natural question is  if the Wodzicki residue can be computed from this algebraic formalism, or equivalently:  
\begin{equation}\label{question}
\boxed{{ \textnormal{is }  \textnormal{res}\,(A) \,\, \textnormal{ information encoded  by the representation  theory of } G\textnormal{?}}}
\end{equation}
In this work we answer \eqref{question} in the affirmative, by proving the validity of  the formula
\begin{equation}\label{res:intro}
    \textnormal{res}(A)=\int\limits_{G}\Vert \sigma_{-n}(x,[\xi]) \Vert_{\mathcal{L}^{(1,\infty)}(\widehat{G})}dx,
\end{equation}
for a positive pseudo-differential operator $A$ of order $-n,$ where  the full symbol $\sigma_{-n}(x,[\xi]),$ $(x,[\xi])\in G\times \widehat{G},$ is the component of order $-n,$ of the matrix-valued symbol $\sigma_A(x,[\xi])$ of $A,$ which is a global mapping on the non-commutative phase space $G\times \widehat{G}$. Here, the $\mathcal{L}^{(1,\infty)}$-norm of the symbol is given by
\begin{equation*}
    \Vert\sigma_{-n}(x,[\xi])\Vert_{\mathcal{L}^{(1,\infty)}(\widehat{G})}:=\lim_{N\rightarrow\infty}\frac{1}{\log N}\sum_{[\xi]:\langle \xi\rangle\leq N}d_{\xi}\textnormal{{Tr}}(|\sigma_{-n}(x,[\xi])|),
\end{equation*}
where the sequence $\langle \xi\rangle, $ $[\xi]\in \widehat{G},$ is formed by the eigenvalues of the operator $(1+\mathcal{L}_G)^{\frac{1}{2}},$ with $\mathcal{L}_G=-\sum_{i=1}^nX_j^2$ being the positive Laplacian on $G,$ see Remark \ref{LaplaceDefi} for details.
Moreover, we compute the Wodzicki residue  for classical pseudo-differential operators in Corollary \ref{WRGroup:2} without the assumption of the positivity of the operator. 
 
The residue formula will be obtained for classical operators as a consequence of  a more general formula (see Corollary \ref{thmmLieGroup:2:2}) that we prove in the setting of subelliptic operators on compact Lie groups \cite{CR20}. Indeed,   in Corollary \ref{Cor:Lidskii}, for a positive subelliptic operator $A,$  we prove the formula
\begin{equation}\label{formulaLidskii2}
    \textnormal{\bf{Tr}}_{\omega}(A)=\int\limits_{G}\Vert \sigma_{-Q}(x,[\xi]) \Vert_{\mathcal{L}^{(1,\infty)}(\widehat{G})}dx=\lim_{N\rightarrow \infty} \frac{1}{\log(N)}\sum_{j=1}^{N}s_{j}(A),
\end{equation}
where $s_{j}(A)$ denotes the sequence of singular values of $A$ with multiplicities taken into account. We refer the reader to Theorem \ref{Residue:Thm} for the Dixmier trace formula  of $A$   without the positivity assumption. In \eqref{formulaLidskii2}, $Q$ denotes the Hausdorff dimension associated to a positive sub-Laplacian $\mathcal{L}=-\sum_{j=1}^kX_j^2,$ associated to a H\"ormander system of vector-fields $\{X_j\}_{j=1}^k$ (see Subsection \ref{Sub:Hor:Cl} for details), and $A$ is a positive operator in the H\"ormander classes associated to  $\mathcal{L},$ see Subsection \ref{Sub:Hor:Cl} or \cite{CR20} for details. Let us remark that \eqref{formulaLidskii2} is a Lidskii type formula, that  relates the Dixmier trace with the eigenvalues of the operator.  An analogous of the Lidskii formula can be obtained in the setting of Dixmier traces following the recent work of   Sukochev,  Zanin et al \cite{GHm:a1}, \cite{suk:a1}, \cite{suk:a2},   \cite{suk:a3}, \cite{suk:a4}.

We conclude the introduction by focusing our attention on two aspects of the analysis on Lie groups that are regarded as quite far apart, and that are related with the validity of \eqref{res:intro}.  One comes from the theories in \cite{Ruz,CR20}
which show that the analysis of pseudo-differential operators using global symbols can be effectively treated when working on compact Lie groups. Indeed,  different  techniques regarding the representation theory of a compact Lie group (where the unitary dual is a discrete set) allow to use generalisations of the Fourier analysis on the torus, or on $\textnormal{SU(2)}.$ In this framework Lidskii type formulae are used to localise the spectrum of an operator and to establish trace formulae in terms of the matrix-valued symbol of an operator (c.f. \cite{del:trace,del:tracetop,dr13a:nuclp,dr:suffkernel,dr13:schatten,dr14a:fsymbsch}). 
Another aspect,  from the point of view of the twisted versions of the periodic analysis   in \cite{Ruz}, or equivalently,  from the twisted structures arising in the non-commutative geometry, e.g.  by deformation of the torus, or by deformation of $\textnormal{SU}(2),$ shows that one can  generalise Wodzicki type results as it was firstly  observed in \cite{CPN} using the notion of a global symbol (with predecessor works on the subject as in  \cite{F}).  Extension of 
the Dixmier trace  in a  more general setting of Banach spaces is available from the works of  Pietsch \cite{Piet:di4}.  For a recent survey on the theory of traces and their history we refer the reader to Pietsch \cite{Piet:b2}. \\

The manuscript is organised as follows: In Section \ref{prel} we first recall basic notions of pseudodifferential calculus on compact Lie groups, subelliptic classes,  determinants on Hilbert spaces and the Dixmier ideal. In Section \ref{SectioDixmier} we study Dixmier traces and regularised determinants. We present there the main results of this work. For instance,   we prove the validity of \eqref{res:intro} in its more general form  for subelliptic operators in Theorem \ref{Residue:Thm}, and consequently for classical elliptic pseudo-differential operators in Corollary \ref{WRGroup:3:3}. In Corollary \ref{WRGroup:2} we compute a formula for the Wodzicki residue of elliptic and classical pseudo-differential operators in terms of their global symbols. Finally, in Section \ref{fremarks} we give some final remarks about the regularised determinant for the Dixmier trace on arbitrary compact manifolds. 

\section{Preliminaries}\label{prel}
In this section we recall some basics of the pseudodifferential calculus on compact Lie groups, the subelliptic H\"ormander classes recently introduced in \cite{CR20}, determinants on Hilbert spaces and the Dixmier ideal.
\subsection{Global H\"ormander classes of pseudo-differential operators}\label{Global:operators:Sm:rho:delta}
First, let us record the notion of the unitary dual $\widehat{G}$ of a compact Lie group $G.$ So,  
let us assume that $\xi$ is a continuous, unitary and irreducible  representation of $G,$ that is
\begin{itemize}
    \item $\xi\in \textnormal{Hom}(G, \textnormal{U}(H_{\xi})),$ for some finite-dimensional vector space $H_\xi\cong \mathbb{C}^{d_\xi},$ i.e. $\xi(xy)=\xi(x)\xi(y)$ and for the  adjoint of $\xi(x),$ $\xi(x)^*=\xi(x^{-1}),$ for every $x,y\in G.$
    \item The map $(x,v)\mapsto \xi(x)v, $ from $G\times H_\xi$ into $H_\xi$ is continuous.
    \item For every $x\in G,$ and $W_\xi\subset H_\xi,$ if $\xi(x)W_{\xi}\subset W_{\xi},$ then $W_\xi=H_\xi$ or $W_\xi=\emptyset.$
\end{itemize} Let $\textnormal{Rep}(G)$ be the set of unitary, continuous and irreducible representations of $G.$ The relation, {\small{
\begin{equation*}
    \xi_1\sim \xi_2\textnormal{ if and only if, there exists } \varkappa\in \textnormal{End}(H_{\xi_1},H_{\xi_2}),\textnormal{ such that }\varkappa\xi_{1}(x)\varkappa^{-1}=\xi_2(x), 
\end{equation*}}}for every $x\in G,$ is an equivalence relation and the unitary dual of $G,$ denoted by $\widehat{G}$ is defined via
$$
    \widehat{G}:={\textnormal{Rep}(G)}/{\sim}.
$$ The compactness of $G$ implies that $\widehat{G}$ is a discrete set. By a suitable changes of basis, we always can assume that every $\xi$ is matrix-valued and that $H_{\xi}=\mathbb{C}^{d_\xi}.$ If a representation $\xi$ is unitary, then $$\xi(G):=  \{\xi(x):x\in G \}$$ is a subgroup (of the group of  matrices $\mathbb{C}^{d_\xi\times d_\xi}$) which is isomorphic to the original group $G$. Thus the homomorphism $\xi$ allows us to represent the compact Lie group $G$ as a group of matrices. This is the motivation for the term `representation'.

The prototype of differential operators of the first order are vector-fields. One can associate first differential operators to vector fields using the exponential mapping as follows. We denote by $\mathfrak{g}$ the Lie algebra of $G$.
\begin{rem}
 Every $X\in\mathfrak{g},$ can be identified with the differential operator $X:C^\infty(G)\rightarrow C^\infty(G)$  defined by
 \begin{equation*}
     (X_{x}f)(x):=\frac{d}{dt}f(x\exp(tX) )|_{t=0},\quad x\in G.
 \end{equation*}
\end{rem}
The Laplacian on a compact Lie group is the model of a differential operator of second order. Its eigenvalues will be used to define the H\"ormander classes on $G$ (see Remark \ref{HC:LG} below). It can be defined as follows.
\begin{rem}[The Laplacian on $G$]\label{LaplaceDefi} Let $G$ be a compact Lie group of dimension $n$
 and let  $\textnormal{ad}:\mathfrak{g}\rightarrow \textnormal{End}(\mathfrak{g})$ be the adjoint representation of its Lie algebra $\mathfrak{g}$. 
With respect to the the Killing form
\begin{equation*}
    B(Y_1,Y_2):=\textnormal{Tr}(\textnormal{ad}(Y_1)\circ \textnormal{ad}(Y_2)),
\end{equation*}on $\mathfrak{g},$ let us choose an orthonormal basis $X_{1},\cdots, X_{n},$ of $\mathfrak{g}.$ The positive Laplacian on $G,$ (Casimir element) is defined via
\begin{equation}
    \mathcal{L}_{G}=-\sum_{j=1}^nX_{j}^2.
\end{equation}The operator $\mathcal{L}_G$ has $L^2(G)$-discrete spectrum $\{\lambda_{[\xi]}\}_{[\xi]\in \widehat{G}},$ and the eigenvalues of the (Bessel potential) operator  $(1+\mathcal{L}_G)^{\frac{1}{2}}$ are given by
\begin{equation}\label{weight}
    \langle \xi\rangle:=(1+\lambda_{[\xi]})^{\frac{1}{2}},
\end{equation} in view of the spectral mapping theorem. The weight function \eqref{weight} will be crucial in defining the H\"ormander classes of symbols.
\end{rem}
\begin{rem} Let $[\xi]\in \widehat{G},$ and $\xi\in [\xi].$  The entry-functions $\xi_{ij}$ of any unitary representation $\xi=(\xi_{ij})_{i,j=1}^{d_\xi}$ span an eigenspace $\mathcal{H}^\xi\subset L^2(G)$ of the Laplacian  and the restriction $\mathcal{L}_G|_{\mathcal{H}^\xi}=\lambda_{[\xi]}I,$ is a multiple of the identity operator on $\mathcal{H}^\xi,$ see Theorem 10.3.13 of \cite{Ruz}. The number $\lambda_{[\xi]}$ is independent of the choice of $\xi\in [\xi].$ Consequently, for any $\xi,$
the entry-functions $\xi_{ij}$   are of $C^\infty$-class. Indeed, they are eigenfunctions of the positive Laplace operator $\mathcal{L}_G$, that is $\mathcal{L}_G\xi_{ij}=\lambda_{[\xi]}\xi_{ij}.$
\end{rem}

We now follow \cite[Chapter 10]{Ruz} to introduce the analysis of operators on the phase space $G\times \widehat{G}.$ Indeed, if $A$ is a continuous linear operator on $C^\infty(G),$  there exists a function \begin{equation}\label{symbol}a:G\times \widehat{G}\rightarrow \cup_{\ell\in \mathbb{N}} \mathbb{C}^{\ell\times \ell},\end{equation} such that for every equivalence class $[\xi]\in \widehat{G},$ $a(x,\xi):=a(x,[\xi])\in \mathbb{C}^{d_\xi\times d_\xi},$ (where $d_\xi$ is the dimension of the continuous, unitary and irreducible  representation $\xi:G\rightarrow \textnormal{U}(\mathbb{C}^{d_\xi})$) and satisfying
\begin{equation}\label{RuzhanskyTurunenQuanti}
    Af(x)=\sum_{[\xi]\in \widehat{G}}d_\xi\textnormal{\textbf{Tr}}[\xi(x)a(x,\xi)\widehat{f}(\xi)],\,\,f\in C^\infty(G).
\end{equation}Here 
\begin{equation*}
    \widehat{f}(\xi)\equiv (\mathscr{F}f)(\xi):=\int\limits_{G}f(x)\xi(x)^*dx\in  \mathbb{C}^{d_\xi\times d_\xi},\,\,\,[\xi]\in \widehat{G},
\end{equation*}is the matrix-valued Fourier transform of $f$ at $\xi=(\xi_{ij})_{i,j=1}^{d_\xi},$ and 
the function $a$ in  \eqref{RuzhanskyTurunenQuanti}  satisfies the identity, 
\begin{equation}
    a(x,\xi)\equiv \widehat{A}(x,\xi)\equiv \sigma_A(x,\xi)=\xi(x)^{*}(A\xi)(x),\,\, A\xi:=(A\xi_{ij})_{i,j=1}^{d_\xi},\,\,\,\,[\xi]\in \widehat{G}.
\end{equation}In general, we refer to the function $a$ as the (global or matrix) {\it{symbol}} of the operator $A.$ 
\begin{rem}
We denote by $\mathscr{S}(\widehat{G}):=\mathscr{F}(C^\infty(G))$ the Schwartz space on the unitary dual. Then the Fourier transform on the group $\mathscr{F}$ is a bijective mapping from $C^\infty(G)$ into $\mathscr{S}(\widehat{G})$ (see \cite[Page 541]{Ruz}), and in terms of the Fourier transform we have
\begin{equation*}
    Af(x)=\mathscr{F}^{-1}[ a(x,\cdot )(\mathscr{F}{f})\,](x),
\end{equation*}for every $f\in C^{\infty}(G).$ In particular, if $a(x,\xi)=I_{d_\xi}$ is the identity matrix in every representation space, $A\equiv I$ is the identity operator on $C^\infty(G),$ and we recover the Fourier inversion formula
\begin{equation*}
    f(x)=\sum_{[\xi]\in \widehat{G}}d_\xi\textnormal{Tr}[\xi(x)\widehat{f}(\xi)],\,\,f\in C^\infty(G).
\end{equation*}
\end{rem}
To classify symbols in the H\"ormander classes  developed in \cite{Ruz}, the notion of {\it{ difference operators}} on the unitary dual, by endowing $\widehat{G}$ with a difference structure, is an instrumental tool.  By following  \cite{RuzhanskyWirth2015},   a difference operator $Q_\xi$ of order $k,$  is defined by
\begin{equation}\label{taylordifferences}
    Q_\xi\widehat{f}(\xi)=\widehat{qf}(\xi),\,[\xi]\in \widehat{G}, 
\end{equation} for all $f\in C^\infty(G),$ for some function $q$ vanishing of order $k$ at the identity $e=e_G.$ We will denote by $\textnormal{diff}^k(\widehat{G})$  the set of all difference operators of order $k.$ For a  fixed smooth function $q,$ the associated difference operator will be denoted by $\Delta_q:=Q_\xi.$ We will choose an admissible collection of difference operators (see e.g. \cite{RuzhanskyWirth2015}),
\begin{equation*}
  \Delta_{\xi}^\alpha:=\Delta_{q_{(1)}}^{\alpha_1}\cdots   \Delta_{q_{(i)}}^{\alpha_{i}},\,\,\alpha=(\alpha_j)_{1\leqslant j\leqslant i}, 
\end{equation*}
where
\begin{equation*}
    \textnormal{rank}\{\nabla q_{(j)}(e):1\leqslant j\leqslant i \}=\textnormal{dim}(G), \textnormal{   and   }\Delta_{q_{(j)}}\in \textnormal{diff}^{1}(\widehat{G}).
\end{equation*}We say that this admissible collection is strongly admissible if 
\begin{equation*}
    \bigcap_{j=1}^{i}\{x\in G: q_{(j)}(x)=0\}=\{e_G\}.
\end{equation*}

\begin{rem}\label{remarkD} A special type of difference operators can be defined by using the unitary representations  of $G.$ Indeed, if $\xi_{0}$ is a fixed irreducible and unitary  representation of $G$, consider the matrix
\begin{equation}
 \xi_{0}(g)-I_{d_{\xi_{0}}}=[\xi_{0}(g)_{ij}-\delta_{ij}]_{i,j=1}^{d_\xi},\, \quad g\in G.   
\end{equation}Then, we associate  to the function 
$
    q_{ij}(g):=\xi_{0}(g)_{ij}-\delta_{ij},\quad g\in G,
$ a difference operator  via
\begin{equation}
    \mathbb{D}_{\xi_0,i,j}:=\mathscr{F}(\xi_{0}(g)_{ij}-\delta_{ij})\mathscr{F}^{-1}.
\end{equation}
If the representation is fixed we omit the index $\xi_0$ so that, from a sequence $\mathbb{D}_1=\mathbb{D}_{\xi_0,j_1,i_1},\cdots, \mathbb{D}_n=\mathbb{D}_{\xi_0,j_n,i_n}$ of operators of this type we define $\mathbb{D}^{\alpha}=\mathbb{D}_{1}^{\alpha_1}\cdots \mathbb{D}^{\alpha_n}_n$, where $\alpha\in\mathbb{N}^n$.
\end{rem}
\begin{rem}[Leibniz rule for difference operators]\label{Leibnizrule} The difference structure on the unitary dual $\widehat{G},$ induced by the difference operators acting on the momentum variable $[\xi]\in \widehat{G},$  implies the following Leibniz rule 
\begin{align*}
    \Delta_{\xi}^{\alpha}(a_{1}a_{2})(x_0,\xi) =\sum_{ |\gamma|,|\varepsilon|\leqslant |\alpha|\leqslant |\gamma|+|\varepsilon| }C_{\varepsilon,\gamma}(\Delta_{\xi}^\gamma a_{1})(x_0,\xi) (\Delta_{\xi}^\varepsilon a_{2})(x_0,\xi), \quad (x_{0},[\xi])\in G\times \widehat{G},
\end{align*} for $a_{1},a_{2}\in C^{\infty}(G, {S}'(\widehat{G})).$ For more details we refer the reader to  \cite{Ruz}.
\end{rem}

\begin{rem}[H\"ormander classes on Lie groups]\label{HC:LG}
If $A\in \Psi^m_{\rho,\delta}(G,\textnormal{loc}),$ $\rho\geqslant   1-\delta,$ the matrix-valued symbol $\sigma_A$ of $A$ satisfies (see \cite{Ruz,RuzhanskyTurunenWirth2014}),
\begin{equation}\label{HormanderSymbolMatrix}
    \Vert {X}_x^\beta \Delta_{\xi}^{\gamma} \sigma_A(x,\xi)\Vert_{\textnormal{op}}\leqslant C_{\alpha,\beta}
    \langle \xi \rangle^{m-\rho|\gamma|+\delta|\beta|}\end{equation} for all $\beta$ and  $\gamma $ multi-indices and all $(x,[\xi])\in G\times \widehat{G}$. Now, if $0\leqslant \delta,\rho\leqslant 1,$
we say that $\sigma_A\in {S}^m_{\rho,\delta}(G\times \widehat{G}),$ if the global symbol inequalities \eqref{HormanderSymbolMatrix}  hold true. So, for $\sigma_A\in {S}^m_{\rho,\delta}(G\times \widehat{G})$ we write $A\in\Psi^m_{\rho,\delta}(G)\equiv\textnormal{Op}({S}^m_{\rho,\delta}(G\times \widehat{G})).$  Observe that 
\begin{equation*}
   \textnormal{Op}({S}^m_{\rho,\delta}(G\times \widehat{G}))= \Psi^m_{\rho,\delta}(G,\textnormal{loc}),\,\,\,0\leqslant \delta<\rho\leqslant 1,\,\rho\geqslant   1-\delta.
\end{equation*}
\end{rem}
An important type of pseudo-differential operators are the elliptic operators. They are invertible modulo smoothing operators as we discuss in the following remark.
\begin{rem}
For  $m\in \mathbb{R},$ and  $0\leqslant \delta<\rho\leqslant 1,$  a symbol  $a=a(x,\xi)\in {S}^{m}_{\rho,\delta}(G\times \widehat{G})$ is elliptic, if   $a(x,\xi)$ is invertible for every $(x,[\xi])\in G\times\widehat{G},$ (with the exception of finitely many equivalence classes $[\xi]$) and if it satisfies
\begin{equation}
   \Vert a(x,\xi)^{-1}\Vert_{\textnormal{op}}\leq C\langle \xi\rangle^{-m},\,\,\,\,\langle\xi \rangle\geq  N,
\end{equation}for some $N\in \mathbb{N}$ large enough. In this case we say that the pseudo-differential operator $A$ is elliptic of order $m.$ The crucial property of elliptic operators is that they are of Fredholm type on $L^2(G).$ Indeed,  by following  \cite{Ruz}, if $0\leq\delta<\rho\leq 1,$ there exists $B\in {S}^{-m}_{\rho,\delta}(G\times \widehat{G}),$ such that $$R_1:=AB-I,R_2:=BA-I\in {S}^{-\infty}(G\times \widehat{G}):=\bigcap_{m\in \mathbb{R}} {S}^m_{\rho,\delta}(G\times \widehat{G}). $$ Indeed, the operators $R_1$ and $R_2$ are compact on $L^2(G).$ This holds in general for closed manifolds 
but the additional condition $\rho\geq 1-\delta$ is required.
\end{rem}

\subsection{H\"ormander classes in the presence of a sub-Riemannian structure}\label{Sub:Hor:Cl}
The H\"ormander classes in \cite{Ruz} are defined using the spectrum of the Laplacian. In \cite{CR20}, these classes were extended by using as measure of decay the spectrum of the sub-Laplacian. This generalisation is a useful tool when studying subelliptic (analytic and spectral) problems as  it was extensively developed in \cite{CR20}. To introduce this calculus let us fix the necessary notations.\\ 

Let $G$ be a compact Lie group  with Lie algebra $\mathfrak{g}.$ Under the identification $\mathfrak{g}\simeq T_{e_G}G,  $ where $e_{G}$ is the identity element of $G,$ let us consider  a system of $C^\infty$-vector fields $X=\{X_1,\cdots,X_k \}\in \mathfrak{g}$. For all $I=(i_1,\cdots,i_\omega)\in \{1,2,\cdots,k\}^{\omega},$ of length $\omega\geqslant   2,$ denote $$X_{I}:=[X_{i_1},[X_{i_2},\cdots [X_{i_{\omega-1}},X_{i_\omega}]\cdots]],$$ and for $\omega=1,$ $I=(i),$ $X_{I}:=X_{i}.$ Let $V_{\omega}$ be the subspace generated by the set $\{X_{I}:|I|\leqslant \omega\}.$ That $X$ satisfies the H\"ormander condition,  means that there exists $\kappa'\in \mathbb{N}$ such that $V_{\kappa'}=\mathfrak{g}.$ Certainly, we consider the smallest $\kappa'$ with this property and we denote it by $\kappa$ which will be  later called the step of the system $X.$ We also say that $X$ satisfies the H\"ormander condition of order $\kappa.$ Note that  the sum of squares
\begin{equation*}
    \mathcal{L}\equiv \mathcal{L}_{X}:=-(X_{1}^2+\cdots +X_{k}^2),
\end{equation*} is a subelliptic operator which by following the usual nomenclature is called the subelliptic Laplacian associated with the family $X.$ For short we refer to $\mathcal{L}$ as the sub-Laplacian.\\

A central notion in the analysis of the sub-Laplacian is that of the Hausdorff dimension, in this case, associated to $\mathcal{L}$. Indeed, for all $x\in G,$ denote by $H_{x}^\omega G$ the subspace of the tangent space $T_xG$ generated by the $X_i$'s and all the Lie brackets  $$ [X_{j_1},X_{j_2}],[X_{j_1},[X_{j_2},X_{j_3}]],\cdots, [X_{j_1},[X_{j_2}, [X_{j_3},\cdots, X_{j_\omega}] ] ],$$ with $\omega\leqslant \kappa.$ The H\"ormander condition can be stated as $H_{x}^\kappa G=T_xG,$ $x\in G.$ We have the filtration
\begin{equation*}
H_{x}^1G\subset H_{x}^2G \subset H_{x}^3G\subset \cdots \subset H_{x}^{\kappa-1}G\subset H_{x}^\kappa G= T_xG,\,\,x\in G.
\end{equation*} In our case,  the dimension of every $H_x^\omega G$ does not depend on $x$ and we write $\dim H^\omega G:=\dim H_{x}^\omega G,$ for any $x\in G.$ So, the Hausdorff dimension can be defined as (see e.g. \cite[p. 6]{HK16}),
\begin{equation}\label{Hausdorff-dimension}
    Q:=\dim(H^1G)+\sum_{i=1}^{\kappa-1} (i+1)(\dim H^{i+1}G-\dim H^{i}G ).
\end{equation} We will use the notation    $\widehat{ \mathcal{M}}$ for  the matrix-valued symbol of the operator $\mathcal{M}:=(1+\mathcal{L})^{\frac{1}{2}},$ and  for every $[\xi]\in \widehat{G},$ and $s\in \mathbb{R}, $
   \begin{equation*}
       \widehat{ \mathcal{M}}(\xi)^{s}:=\textnormal{diag}[(1+\nu_{ii}(\xi)^2)^{\frac{s}{2}}]_{1\leqslant i\leqslant d_\xi},
   \end{equation*} where $\widehat{\mathcal{L}}(\xi)=:\textnormal{diag}[\nu_{ii}(\xi)^2]_{1\leqslant i\leqslant d_\xi}$ is the symbol of the sub-Laplacian $\mathcal{L}$ at $[\xi].$\\
   
In order to introduce the H\"ormander classes associated to the sub-Laplacian $\mathcal{L},$ we require a suitable class of differential operators $\partial_{X}^{(\beta)}$, compatible with the difference operators and the remainders in the Taylor expansion formula on compact Lie groups. We introduce such a family in the following lemma (see Lemma 4.2 of \cite{CR20}).
  \begin{lem}[Global Taylor Series on compact Lie groups]\label{Taylorseries} Let $G$ be a compact Lie group of dimension $n.$ Let us consider an strongly admissible admissible collection of difference operators $\mathfrak{D}=\{\Delta_{q_{(j)}}\}_{1\leqslant j\leqslant n}$, which means that 
\begin{equation*}
    \textnormal{rank}\{\nabla q_{(j)}(e):1\leqslant j\leqslant n \}=n, \,\,\,\bigcap_{j=1}^{n}\{x\in G: q_{(j)}(x)=0\}=\{e_G\}.
\end{equation*}Then there exists a basis $X_{\mathfrak{D}}=\{X_{1,\mathfrak{D}},\cdots ,X_{n,\mathfrak{D}}\}$ of $\mathfrak{g},$ such that $X_{j,\mathfrak{D}}q_{(k)}(\cdot^{-1})(e_G)=\delta_{jk}.$ Moreover, by using the multi-index notation $\partial_{X}^{(\beta)}=\partial_{X_{i,\mathfrak{D}}}^{\beta_1}\cdots \partial_{X_{n,\mathfrak{D}}}^{\beta_n}, $ $\beta\in\mathbb{N}_0^n,$
where $$\partial_{X_{i,\mathfrak{D}}}f(x):=  \frac{d}{dt}f(x\exp(tX_{i,\mathfrak{D}}) )|_{t=0},\,\,f\in C^{\infty}(G),$$ and denoting
for every $f\in C^{\infty}(G)$
\begin{equation*}
    R_{x,N}^{f}(y):=f(xy)-\sum_{|\alpha|<N}q_{(1)}^{\alpha_1}(y^{-1})\cdots q_{(n)}^{\alpha_n}(y^{-1})\partial_{X}^{(\alpha)}f(x),
\end{equation*} we have that 
\begin{equation*}
    | R_{x,N}^{f}(y)|\leqslant C|y|^{N}\max_{|\alpha|\leqslant N}\Vert \partial_{X}^{(\alpha)}f\Vert_{L^\infty(G)}.
\end{equation*}The constant $C>0$ is dependent on $N,$ $G$ and $\mathfrak{D},$ but not on $f\in C^\infty(G).$ Also, we have that $\partial_{X}^{(\beta)}|_{x_1=x}R_{x_1,N}^{f}=R_{x,N}^{\partial_{X}^{(\beta)}f}$ and 
\begin{equation*}
    | \partial_{X}^{(\beta)}|_{y_1=y}R_{x,N}^{f}(y_1)|\leqslant C|y|^{N-|\beta|}\max_{|\alpha|\leqslant N-|\beta|}\Vert \partial_{X}^{(\alpha+\beta)}f\Vert_{L^\infty(G)},
\end{equation*}provided that $|\beta|\leqslant N.$
 \end{lem}
Now, we introduce the H\"ormander classes for the sub-Laplacian as developed in \cite{CR20}.
\begin{defn}[Subelliptic H\"ormander classes]\label{contracted''}
   Let $G$ be a compact Lie group and let $0\leqslant \delta,\rho\leqslant 1.$ Let us consider a sub-Laplacian $\mathcal{L}=-(X_1^2+\cdots +X_k^2)$ on $G,$ where the system of vector fields $X=\{X_i\}_{i=1}^{k}$ satisfies the H\"ormander condition of step $\kappa$. A symbol $\sigma$ belongs to the subelliptic H\"ormander class ${S}^{m,\mathcal{L}}_{\rho,\delta}(G\times \widehat{G})$ if it satisfies the following  symbol inequalities
   \begin{equation}\label{InIC}
      p_{\alpha,\beta,\rho,\delta,m,\textnormal{left}}(\sigma)':= \sup_{(x,[\xi])\in G\times \widehat{G} }\Vert \widehat{ \mathcal{M}}(\xi)^{(\rho|\alpha|-\delta|\beta|-m)}\partial_{X}^{(\beta)} \Delta_{\xi}^{\alpha}\sigma(x,\xi)\Vert_{\textnormal{op}} <\infty,\,\forall\alpha,\beta,
   \end{equation}
   and 
   \begin{equation}\label{InIIC}
      p_{\alpha,\beta,\rho,\delta,m,\textnormal{right}}(\sigma)':= \sup_{(x,[\xi])\in G\times \widehat{G} }\Vert (\partial_{X}^{(\beta)} \Delta_{\xi}^{\alpha} \sigma(x,\xi) ) \widehat{ \mathcal{M}}(\xi)^{(\rho|\alpha|-\delta|\beta|-m)}\Vert_{\textnormal{op}} <\infty,\,\forall\alpha,\beta.
   \end{equation}  The family of seminorms
   \begin{equation}
      p_{\alpha,\beta,\rho,\delta,m,\textnormal{right}}(\sigma)', \,p_{\alpha,\beta,\rho,\delta,m,\textnormal{right}}(\sigma)'
   \end{equation}endow the class ${S}^{m,\mathcal{L}}_{\rho,\delta}(G\times \widehat{G})$  with a natural structure of Fr\'echet space. By denoting
   \begin{equation}
     \Psi^{m,\mathcal{L}}_{\rho,\delta}(G\times \widehat{G}):=  \{A: \sigma_A\in S^{m,\mathcal{L}}_{\rho,\delta}(G\times \widehat{G}) \},
   \end{equation} and by defining 
   \begin{equation}
       p_{\alpha,\beta,\rho,\delta,m,\textnormal{right}}(A)':=p_{\alpha,\beta,\rho,\delta,m,\textnormal{right}}(\sigma_A)',\quad p_{\alpha,\beta,\rho,\delta,m,\textnormal{left}}(A)':=p_{\alpha,\beta,\rho,\delta,m,\textnormal{left}}(\sigma_A)',
   \end{equation}the operator class $\Psi^{m,\mathcal{L}}_{\rho,\delta}(G\times \widehat{G})$ also has a structure of Fr\'echet space.
  \end{defn}
  In the next theorem we summarise some of the fundamental properties of the subelliptic calculus in \cite{CR20} that we will use in our further analysis.
\begin{thm}\label{calculus} Let $0\leqslant \delta<\rho\leqslant 1,$ and let us denote $\Psi^{m,\mathcal{L}}_{\rho,\delta}:=\textnormal{Op}({S}^{m,\mathcal{L}}_{\rho,\delta}(G\times \widehat{G})),$ for every $m\in \mathbb{R}.$ Then,
\begin{itemize}
    \item [(i)] The mapping $A\mapsto A^{*}:\Psi^{m,\mathcal{L}}_{\rho,\delta}\rightarrow \Psi^{m,\mathcal{L}}_{\rho,\delta}$ is a continuous linear mapping between Fr\'echet spaces and  the  symbol of $A^*,$ $\sigma_{A^*}(x,\xi)$ satisfies the asymptotic expansion,
 \begin{equation*}
    \widehat{A^{*}}(x,\xi)\sim \sum_{|\alpha|= 0}^\infty\Delta_\xi^{\alpha}\partial_{X}^{(\alpha)} (\widehat{A}(x,\xi)^{*}).
 \end{equation*} This means that, for every $N\in \mathbb{N},$ and all $\ell\in \mathbb{N},$
\begin{equation*}
   \Small{ \Delta_{\xi}^{\alpha_\ell}\partial_{X}^{(\beta)}\left(\widehat{A^{*}}(x,\xi)-\sum_{|\alpha|\leqslant N}\Delta_{\xi}^\alpha\partial_{X}^{(\alpha)} (\widehat{A}(x,\xi)^{*}) \right)\in {S}^{m-(\rho-\delta)(N+1)-\rho\ell+\delta|\beta|,\mathcal{L}}_{\rho,\delta}(G\times\widehat{G}) },
\end{equation*} where $|\alpha_\ell|=\ell.$
\item [(ii)] The mapping $(A_1,A_2)\mapsto A_1\circ A_2: \Psi^{m_1,\mathcal{L}}_{\rho,\delta}\times \Psi^{m_2,\mathcal{L}}_{\rho,\delta}\rightarrow \Psi^{m_3,\mathcal{L}}_{\rho,\delta}$ is a continuous bilinear mapping between Fr\'echet spaces, and the symbol of $A=A_{1}\circ A_2,$ satisfies the asymptotic expansion,
\begin{equation*}
    \sigma_A(x,\xi)\sim \sum_{|\alpha|= 0}^\infty(\Delta_{\xi}^\alpha\widehat{A}_{1}(x,\xi))(\partial_{X}^{(\alpha)} \widehat{A}_2(x,\xi)),
\end{equation*}this means that, for every $N\in \mathbb{N},$ and all $\ell \in\mathbb{N},$
\begin{align*}
    &\Delta_{\xi}^{\alpha_\ell}\partial_{X}^{(\beta)}\left(\sigma_A(x,\xi)-\sum_{|\alpha|\leqslant N}  (\Delta_\xi^{\alpha}\widehat{A}_{1}(x,\xi))(\partial_{X}^{(\alpha)} \widehat{A}_2(x,\xi))  \right)\\
    &\hspace{2cm}\in {S}^{m_1+m_2-(\rho-\delta)(N+1)-\rho\ell+\delta|\beta|,\mathcal{L}}_{\rho,\delta}(G\times \widehat{G}),
\end{align*}for every  $\alpha_\ell \in \mathbb{N}_0^n$ with $|\alpha_\ell|=\ell.$
\item [(iii)] For  $0\leqslant \delta< \rho\leqslant    1,$  (or for $0\leq \delta\leq \rho\leq 1,$ $\delta<1/{\kappa}$) let us consider a continuous linear operator $A:C^\infty(G)\rightarrow\mathscr{D}'(G)$ with symbol  $\sigma\in {S}^{0,\mathcal{L}}_{\rho,\delta}(G\times \widehat{G})$. Then $A$ extends to a bounded operator from $L^2(G)$ to  $L^2(G).$ 
\end{itemize}
\end{thm}
Now, we introduce the family of subelliptic pseudo-differential operators that we will investigate in the setting of the Dixmier trace.
\begin{defn}\label{IesTParametrix} Let $m\in \mathbb{R},$ and let $0\leqslant \delta<\rho\leqslant 1.$  Let  $a=a(x,\xi)\in {S}^{m,\mathcal{L}}_{\rho,\delta}(G\times \widehat{G}).$  Assume also that $a(x,\xi)$ is invertible for every $(x,[\xi])\in G\times\widehat{G},$ and satisfies
\begin{equation}\label{Iesparametrix}
   \sup_{(x,[\xi])\in G\times \widehat{G}} \Vert \widehat{\mathcal{M}}(\xi)^{m}a(x,\xi)^{-1}\Vert_{\textnormal{op}}<\infty.
\end{equation} Then we will say that the operator $A$ associated with $a(\cdot,\cdot)$ is $\mathcal{L}$-elliptic.
\end{defn}

\begin{rem}
Let  $A$ be an $\mathcal{L}$-elliptic  operator. Then, Proposition 7.1 in \cite[Page 77]{CR20} shows that there exists $B\in {S}^{-m,\mathcal{L}}_{\rho,\delta}(G\times \widehat{G}),$ such that $AB-I,BA-I\in {S}^{-\infty,\mathcal{L}}(G\times \widehat{G}). $ Moreover, the symbol of $B$ satisfies the following asymptotic expansion
\begin{equation}
    \widehat{B}(x,\xi)\sim \sum_{N=0}^\infty\widehat{B}_{N}(x,\xi),\,\,\,(x,[\xi])\in G\times \widehat{G},
\end{equation}where $\widehat{B}_{N}(x,\xi)\in {S}^{-m-(\rho-\delta)N,\mathcal{L}}_{\rho,\delta}(G\times \widehat{G})$ obeys to the recursive  formula
\begin{equation}
    \widehat{B}_{N}(x,\xi)=-a(x,\xi)^{-1}\left(\sum_{k=0}^{N-1}\sum_{|\gamma|=N-k}(\Delta_{\xi}^\gamma a(x,\xi))(\partial_{X}^{(\gamma)}\widehat{B}_{k}(x,\xi))\right),\,\,N\geqslant 1,
\end{equation}with $ \widehat{B}_{0}(x,\xi)=a(x,\xi)^{-1}.$
\end{rem}

\subsection{Determinants on Hilbert spaces}
Since we are interested in the study of the Dixmier trace and its regularised determinant we will require some preliminary basic notions. First, we recall the definition of Schatten-von Neumann ideals on Hilbert spaces. A reason for this is that we will use operators $T$ of Dixmier type, where $T^p$ is of class trace. 

Let $\Hcal$ be a complex separable Hilbert space endowed with an inner product denoted 
by $\p{\cdot,\cdot}$, and let $T:\Hcal\rightarrow \Hcal$ be a linear compact operator. Let $T^*:\Hcal\rightarrow \Hcal$ be the adjoint of  $T$. Then the linear operator $(T^*T)^\half:\Hcal\rightarrow \Hcal$ is positive and compact. Let $(\psi_k)_k$ be an orthonormal basis for $\Hcal$ consisting of eigenvectors of $|T|=(T^*T)^\half$, and let $s_k(T)$ be the eigenvalue corresponding to the eigenvector 
$\psi_k$, $k=1,2,\dots$. The non-negative numbers $s_k(T)$, $k=1,2,\dots$, are called the {\em singular values} of $T:\Hcal\rightarrow \Hcal$. 
If 
$$
\sum_{k=1}^{\infty} s_k(T)<\infty ,
$$
then the linear operator $T:\Hcal\rightarrow \Hcal$ is said to be in the {\em trace class} $S_1(\mathcal{H})$. It can be shown that  $S_1(\Hcal)$ is a Banach space in which the norm $\|\cdot\|_{S_1}$ is given by 
$$
\|T\|_{S_1}= \sum_{k=1}^{\infty} s_k(T),\,T\in S_1(\mathcal{H}),
$$
with multiplicities counted.
Let $T:\Hcal\rightarrow \Hcal$ be an operator in $S_1(\Hcal)$ and let  $(\phi_k)_k$ be any orthonormal basis for $\Hcal$. Then, the series $\sum\limits_{k=1}^{\infty} \p{T\phi_k,\phi_k}$   is absolutely convergent and the sum is independent of the choice of the orthonormal basis $(\phi_k)_k$. Thus, we can define the trace $\Tr(T)$ of any linear operator
$T:\Hcal\rightarrow \Hcal$ in $S_1(\mathcal{H})$ by 
$$
\Tr(T):=\sum_{k=1}^{\infty}\p{T\phi_k,\phi_k},
$$
where $\{\phi_k: k=1,2,\dots\}$ is any orthonormal basis for $\Hcal$. If the singular values
are square-summable $T$ is called a {\em Hilbert-Schmidt} operator. It is clear that every trace class operator is a Hilbert-Schmidt operator.

More generally, if $1\leq p<\infty$ and the sequence of singular values is $p$-summable, then $T$ 
is said to belong to the Schatten-von Neumann class  ${S}_p(\Hcal)$. Indeed,  it is well known that each ${S}_p(\Hcal)$ is an ideal in $\mathcal{L}(\Hcal)$. If $1\leq p <\infty$, a norm is associated to ${S}_p(\Hcal)$ by
\[
\|T\|_{S_p}=\left(\sum\limits_{k=1}^{\infty}(s_k(T))^p\right)^{\frac{1}{p}}.
\] 
If $1\leq p<\infty$ 
the class $S_p(\Hcal)$ becomes a Banach space endowed by the norm $\|T\|_{S_p}$. If $p=\infty$ we define $S_{\infty}(\Hcal)$ as the class of bounded linear operators on $H$, with 
$\|T\|_{S_\infty}:=\|T\|_{\textnormal{op}}$, the operator norm.

\begin{rem}
We note  that  for a positive compact operator $A$ on $\mathcal{H}$ and  $1\leq p<\infty,$ 
$A^{p}\in S_{1}(\mathcal{H}),$ if and only $A\in S_{p}(\mathcal{H}),$ 
\end{rem}
For operators in the class ${S}_1(\mathcal{H})$ we have the following determinant formula, see  \cite[Page 19]{goh:trace}.

\begin{thm}[Plemelj-Smithies formula]\label{thm0a}
If $T\in S_1(\Hcal)$ one has
\[\det(I+\lambda T)=\exp\left(\sum\limits_{m=1}^{\infty}\frac{(-1)^{m+1}}{m}\Tr (T^m)\lambda ^m\right),  \]
for $\lambda\in \mathbb{C}$ with  $|\lambda|$ small enough. \\
\end{thm}

\subsection{Determinant of invariant operators on Hilbert spaces}\label{sec3}
In this section we present the determinant notion in the setting of invariant operators on Hilbert spaces as introduced in \cite{dr14a:fsymbsch}. We record the following theorem:
\begin{thm}\label{THM:inv-rem}
Let $\Hcal$ be a complex Hilbert space and let $\Hcal^{\infty}\subset \Hcal$ be a dense
linear subspace of $\Hcal$. Let $\{d_{j}\}_{j\in\N_{0}}\subset\N$ and let
$\{e_{j}^{k}\}_{j\in\N_{0}, 1\leq k\leq d_{j}}$ be an
orthonormal basis of $\Hcal$ such that
$e_{j}^{k}\in \Hcal^{\infty}$ for all $j$ and $k$. Let $H_{j}:={\rm span} \{e_{j}^{k}\}_{k=1}^{d_{j}}$,
and let $P_{j}:\Hcal\to H_{j}$ be the orthogonal projection.
For $f\in\Hcal$, we denote $\widehat{f}(j,k):=(f,e_{j}^{k})_{\Hcal}$ and let
$\widehat{f}(j)\in \ce^{d_{j}}$ denote the column of $\widehat{f}(j,k)$, $1\leq k\leq d_{j}.$
Let $T:\Hcal^{\infty}\to \Hcal$ be a linear operator.
Then the following
conditions are equivalent:
\begin{itemize}
\item[(A)] For each $j\in\ene_0$, we have $T(H_j)\subset H_j$. 
\item[(B)] For each $\ell\in\ene_0$ there exists a matrix 
$\sigma_{T}(\ell)\in\ce^{d_{\ell}\times d_{\ell}}$ such that for all $e_j^k$ 
$$
\widehat{Te_j^k}(\ell,m)=\sigma_{T}(\ell)_{mk}\delta_{j\ell}.
$$
\item[(C)]  For each $\ell\in\ene_0 $ there exists a matrix 
$\sigma_{T}(\ell)\in\ce^{d_{\ell}\times d_{\ell}}$ such that
 \[\widehat{Tf}(\ell)=\sigma_{T}(\ell)\widehat{f}(\ell)\]
 for all $f\in\Hcal^{\infty}.$
\end{itemize}

The matrices $\sigma_{T}(\ell)$ in {\rm (B)} and {\rm (C)} coincide.

The equivalent properties {\rm (A)--(C)} follow from the condition 
\begin{itemize}
\item[(D)] For each $j\in\ene_0$, we have
$TP_j=P_jT$ on $\Hcal^{\infty}$.
\end{itemize}
If, in addition, $T$ extends to a bounded operator
$T\in{\mathscr L}(\Hcal)$ then {\rm (D)} is equivalent to {\rm (A)--(C)}.
\end{thm} 

Under the assumptions of Theorem \ref{THM:inv-rem}, we have the direct sum 
decomposition
\begin{equation}\label{EQ:sum}
\Hcal = \bigoplus_{j=0}^{\infty} H_{j},\quad H_{j}={\rm span} \{e_{j}^{k}\}_{k=1}^{d_{j}},
\end{equation}
and we have $d_{j}=\dim H_{j}.$
The two main applications that we will consider correspond to $\Hcal=L^{2}(M)$ for a
compact manifold $M$ with $H_{j}$ being the eigenspaces of an elliptic 
pseudo-differential operator $E$, or with $\Hcal=L^{2}(G)$ for a compact Lie group
$G$ with $H_{j}=\textrm{span}\{\xi_{km}\}_{1\leq k,m\leq d_{\xi}}$ for a
unitary irreducible representation $\xi\in[\xi_{j}]\in\widehat{G}$. The difference
is that in the first case we will have that the eigenvalues of $E$ corresponding
to $H_{j}$'s are all distinct, while in the second case the eigenvalues of the Laplacian
on $G$ for which $H_{j}$'s are the eigenspaces, may coincide.
\begin{defn}
An operator $T$ satisfying any of
the equivalent properties (A)--(C) in
Theorem \ref{THM:inv-rem}, will be called an {\em invariant operator}, or
a {\em Fourier multiplier relative to the decomposition
$\{H_{j}\}_{j\in\N_{0}}$} in \eqref{EQ:sum}.
If the collection $\{H_{j}\}_{j\in\N_{0}}$
is fixed once and for all, we can just say that $T$ is {\em invariant}
or a {\em Fourier multiplier}.

\end{defn}

The family of matrices $\sigma$ will be
called the {\em matrix symbol of $T$ relative to the partition $\{H_{j}\}$ and to the
basis $\{e_{j}^{k}\}$.}
It is an element of the space $\Sigma$ defined by
\begin{equation}\label{EQ:Sigma1}
\Sigma=\{\sigma:\N_{0}\ni\ell\mapsto\sigma(\ell)\in \ce^{d_{\ell}\times d_{\ell}}\}.
\end{equation}
In this section we will investigate the concept of determinant on embedded in the sense of \cite{goh:trace}   
for the notion of invariant operators introduced in \cite{dr14a:fsymbsch}. \\

In view of  Theorem \ref{THM:inv-rem} we have the following determinant formula for $E$-invariant operators that later on will be used to establish a determinant formula in terms of the Dixmier trace (for the proof, see \cite{CDR21}).   
\begin{thm}\label{thm1a}  If $T\in S_1(\Hcal)$ is invariant, then
\[\det(I+\lambda T)=\exp\left(\sum\limits_{m=1}^{\infty}\frac{(-1)^{m+1}}{m} \lambda ^m\sum\limits_{\ell=0}^{\infty} \Tr(\sigma_{T}(\ell)^m)  \right),  \]
for $\lambda\in \mathbb{C}$ with  $|\lambda|$ small enough. 
\end{thm}

\subsection{The Dixmier ideal on Hilbert spaces} By following Connes \cite{Connes94}, if $\mathcal{H}$ is a Hilbert space (we are interested in $\mathcal{H}=L^2(M)$ where $M$ is a closed manifold of dimension $n$), the class $\mathcal{L}^{(1,\infty)}(\mathcal{H})$ consists of those compact linear operators $A$ on $\mathcal{H}$  satisfying
\begin{equation}
\sum_{1\leq n\leq N}s_{n}(A)=O(\log(N)),\,\,\,N\rightarrow \infty,
\end{equation}
where $\{s_{n}(A)\}$ denotes the sequence of singular values of $A$,  i.e. the square roots of the eigenvalues of the positive-definite self-adjoint operator $A^\ast A.$ 
So, $\mathcal{L}^{(1,\infty)}(\mathcal{H})$ is endowed with the norm
\begin{equation}\label{dixmier}
\Vert A \Vert_{\mathcal{L}^{(1,\infty)}(\mathcal{H})}=\sup_{N\geq 2}\frac{1}{\log(N)}\sum_{1\leq n\leq N}s_{n}(A).
\end{equation}
For an operator $A\in \mathcal{L}^{(1,\infty)}(\mathcal{H})$ the sequence
\begin{equation}
    \alpha_{N}(A):=\frac{1}{\log(N)}\sum_{j=1}^{N}s_{j}(A) ,\quad N\in \mathbb{N},
\end{equation}is thus bounded. If the limit of the sequence $\{\alpha_{N}(A)\}$ when $N\rightarrow \infty,$ does exist  we will define
\begin{align*}
    \textnormal{\bf Tr}_{\textnormal{Dix}}(A):=\lim_{N\rightarrow \infty} \frac{1}{\log(N)}\sum_{j=1}^{N}s_{j}(A).
\end{align*}

We recall the main result in Comptes Rendus's note of Dixmier  \cite{Dixmier} regarding the existence of a trace functional different from
the spectral trace defined on the ideal  $\mathcal{L}^{(1,\infty)}(\mathcal{H})$ which contains the ideal of trace class operators $S_1(\mathcal{H})$.

\begin{thm}[Dixmier   \cite{Dixmier}]\label{Dixm:Theorem} Let $\omega \in \ell^{\infty}(\mathbb{N}\setminus\{ 0\})^{*}$ be a bounded functional satisfying the following conditions:

\begin{itemize}
    \item[(D1):] $\omega$ is a positive linear functional that satisfies $\omega(1,1,\cdots, 1,\cdots)=1.$  
    \item[(D2):]  $\omega((a_n)_{n\in \mathbb{N}\setminus\{ 0\}})=0,$ for a sequence $(a_n)_{n\in \mathbb{N}\setminus\{ 0\}}\in \ell^{\infty}(\mathbb{N}\setminus\{ 0\})$ if $$ \lim_{N\rightarrow \infty}a_{n}=0.$$ 
    \item[(D3):] For any  $(a_n)_{n\in \mathbb{N}\setminus\{ 0\}}\in \ell^{\infty}(\mathbb{N}\setminus\{ 0\}),$
    $$ \omega(a_1,a_2,\cdots ,a_n,\cdots )= \omega(a_1,a_1,a_2,a_2,\cdots ,a_n,a_n,\cdots ). $$
\end{itemize} For a compact positive operator $A$ in $\mathcal{L}^{(1,\infty)}(\mathcal{H})$ we set
\begin{equation}
    \textnormal{\bf Tr}_\omega(A):=\omega \left(\left\{\frac{1}{\log(N)}\sum_{j=1}^{N}s_{j}(A)\right\}_{N\geq 1} \right).
\end{equation} Then, $ \textnormal{\bf Tr}_\omega$ extends by linearity to a trace on $\mathcal{L}^{(1,\infty)}(\mathcal{H}),$ and 
\begin{equation}\label{TrD}
     \textnormal{\bf Tr}_\omega(A)= \textnormal{\bf Tr}_{\textnormal{Dix}}(A):=\lim_{N\rightarrow \infty} \frac{1}{\log(N)}\sum_{j=1}^{N}s_{j}(A),
\end{equation}provided that the  limit in the right hand side exists. In this case, the value of $ \textnormal{\bf Tr}_\omega(A)$ is independent of the choice of $\omega.$ Moreover, $\textnormal{\bf Tr}_\omega(A)=0,$ if $A\in S_{1}(\mathcal{H}).$
\end{thm}
\begin{rem}\label{Deco:remark} Let us describe the extension of $\textnormal{\bf Tr}_\omega$ from the positive elements in $\mathcal{L}^{(1,\infty)}(\mathcal{H})$ to general compact operators in this ideal, see \cite{Connes94}.
We use the decomposition of $A$ into its real and imaginary part,
\begin{equation*}
    \textnormal{Re}(A):=\frac{A+A^*}{2},\,\, \textnormal{Im}(A):=\frac{A-A^*}{2i},
\end{equation*}and the decomposition of $\textnormal{Re}(A)$ and $\textnormal{Im}(A)$ into their positive and negative parts,
\begin{eqnarray*}
  \textnormal{Re}(A)^{+}:=\frac{\textnormal{Re}(A)+|\textnormal{Re}(A)|}{2},\,\, \textnormal{Re}(A)^{-}:=\frac{|\textnormal{Re}(A)|-\textnormal{Re}(A)}{2},
\end{eqnarray*}
and 
\begin{eqnarray*}
  \textnormal{Im}(A)^{+}:=\frac{\textnormal{Im}(A)+|\textnormal{Im}(A)|}{2},\,\, \textnormal{Im}(A)^{-}:=\frac{|\textnormal{Im}(A)|-\textnormal{Im}(A)}{2}.
\end{eqnarray*}Now, the operator $A$ can be written as
\begin{align*}
    A&= \textnormal{Re}(A)+i\textnormal{Im}(A)\\
    &=\left(\textnormal{Re}(A)^{+}-\textnormal{Re}(A)^{-}\right)+i\left(\textnormal{Im}(A)^{+}-\textnormal{Im}(A)^{-}\right).
\end{align*}So, by the linearity of the Dixmier trace $\textnormal{Tr}_\omega$ we have
\begin{align*}
    \textnormal{Tr}_\omega (A)&= \textnormal{Tr}_\omega(\textnormal{Re}(A))+i\textnormal{Tr}_\omega(\textnormal{Im}(A))\\
    &=\left(\textnormal{Tr}_\omega(\textnormal{Re}(A)^{+})-\textnormal{Tr}_\omega(\textnormal{Re}(A)^{-})\right)+i\left(\textnormal{Tr}_\omega(\textnormal{Im}(A)^{+})-\textnormal{Tr}_\omega(\textnormal{Im}(A)^{-})\right).
\end{align*}
\end{rem}

\begin{defn}\label{def:admisible} We denote by 
\begin{equation}
    \mathcal{L}^{(1,\infty)}_\omega(\mathcal{H}):=\{A\in \mathcal{L}^{(1,\infty)}(\mathcal{H}):\textnormal{\bf{Tr}}_{\textnormal{Dix}}(A)<\infty \}.
\end{equation} the sub-class of operators in  $A\in \mathcal{L}^{(1,\infty)}(\mathcal{H}),$ such that the limit in the right-hand side of \eqref{TrD} exists, and in view of Dixmier's Theorem \ref{Dixm:Theorem}, we have
$$ \textnormal{\bf Tr}_\omega(A)=\lim_{N\rightarrow \infty} \frac{1}{\log(N)}\sum_{j=1}^{N}s_{j}(A),  $$ for all $A\in \mathcal{L}^{(1,\infty)}_{\omega}(\mathcal{H}).$
\end{defn}
\begin{rem}\label{MainRemark}
A remarkable connection between the Dixmier functional and the usual trace on the ideal  $S_{1}(\mathcal{H})$ is the following formula due to Connes and Moscovici
\begin{equation}\label{dixmier2}
\textnormal{\bf{Tr}}_{\omega}(A)=\lim_{p\rightarrow 1^{+}}(p-1)\textnormal{\bf{Tr}}(A^{p}),\,\,A\geq 0,
\end{equation}provided that $A^{p}\in S_{1}(\mathcal{H}),$ for all $1< p<\infty,$ (cf.  \cite[Page 114]{Sukochev}). Indeed, in view of Proposition 4 of \cite[Page 313]{Connes94}, for $A\geq 0,$ the existence of the limit in the right-hand side of \eqref{dixmier2} is equivalent to the fact that $A\in \mathcal{L}^{(1,\infty)}_\omega(H).$
\end{rem}
We finish this subsection by introducing the regularised determinant for the Dixmier trace, which can be obtained from the Plemelj-Smithies formula in Theorem \ref{thm0a}. First, we prove the following lemma.
\begin{lem}\label{lemmadix}
Let $T$ be a positive compact operator on a Hilbert space $H,$ and let us assume that $T\in \mathcal{L}^{(1,\infty)}(\mathcal{H}).$ If $T^{p}\in S_{1}(\mathcal{H}),$  for all  $1< p<\infty$,  then the function 
\begin{equation}
    \textnormal{Det}_{p,\omega,T}(\lambda):= (\textnormal{Det}(1+\lambda T^p))^{p-1},
\end{equation} is  analytic  for $\lambda\in \mathbb{C}$ with  $|\lambda|$ small enough. Moreover, we have
\begin{equation}\label{32:ext}
   \textnormal{Det}_{\omega,T}(\lambda):=\lim_{p\rightarrow 1^{+}} \textnormal{Det}_{p,\omega,T}(\lambda)=\exp(\textnormal{\bf{Tr}}_\omega(T)\lambda),
\end{equation} for $|\lambda|$ small enough. In particular, $$\textnormal{Det}_{\omega,T}'(0):=\frac{d}{d\lambda}\textnormal{Det}_{\omega,T}(\lambda)|_{\lambda=0}=\textnormal{\bf{Tr}}_\omega(T).$$
\end{lem}
\begin{proof} Observe that in view of Theorem \ref{Dixm:Theorem},  $\textnormal{\bf{Tr}}_{\omega}(T^{p})=0,$ for all  $1< p<\infty$,  since in this case $T^{p}\in S_{1}(\mathcal{H}).$ 
In view of Theorem \ref{thm0a}, the function $\textnormal{Det}(1+\lambda T^p)$ is analytic for $\lambda\in \mathbb{C}$ with  $|\lambda|$ small enough. So, $\textnormal{Det}_{p,\omega,T}(\lambda)$ also is analytic in a small disc centered at zero and one has
\[(\det(I+\lambda T^{p}))^{p-1}=\left(\exp\left(\sum\limits_{m=1}^{\infty}\frac{(-1)^{m+1}}{m}\Tr (T^{pm})\lambda ^m\right)\right)^{p-1},  \]
for $\lambda\in \mathbb{C}$ with  $|\lambda|$ small enough. Because of 
\begin{align*}
    (\det(I+\lambda T^{p}))^{p-1}&=\exp\left({(p-1)}\left(\sum\limits_{m=1}^{\infty}\frac{(-1)^{m+1}}{m}\Tr (T^{pm})\lambda ^m\right)\right)\\
    &=\exp\left(\sum\limits_{m=1}^{\infty}\frac{(-1)^{m+1}}{m} (p-1)\Tr (T^{pm})\lambda ^m\right),
\end{align*} and by using \eqref{dixmier2} we see  that
\begin{align*}
     \lim_{p\rightarrow 1^+}(\det(I+\lambda T^{p}))^{p-1} &=\exp\left(\sum\limits_{m=1}^{\infty}\frac{(-1)^{m+1}}{m} \lim_{p\rightarrow 1^+} (p-1)\Tr (T^{pm})\lambda ^m\right)\\
     &=\exp\left(\sum\limits_{m=1}^{\infty}\frac{(-1)^{m+1}}{m} \textnormal{\bf{Tr}}_{\omega}(T^{m})\lambda ^m\right).
\end{align*}The assumption  that $T^{p}\in S_{1}(\mathcal{H}),$  for all  $1< p<\infty$, implies that  $\textnormal{\bf{Tr}}_{\omega}(T^{m})=0,$ for all $m>1,$ and that
   $$ \textnormal{Det}_{\omega,T}(\lambda)= \exp\left(\sum\limits_{m=1}^{\infty}\frac{(-1)^{m+1}}{m} \textnormal{\bf{Tr}}_{\omega}(T^{m})\lambda ^m\right)=\exp(\lambda \textnormal{\bf{Tr}}_{\omega}(T)).  $$ Thus, we end the proof.
\end{proof}
In view of Lemma \ref{lemmadix} let us introduce the following definition.
\begin{defn} Let $\mathcal{L}^{(1,\infty)}_{adm}(\mathcal{H})$ be the family of the  positive compact operators $T$ on  $H,$ such that  $T\in \mathcal{L}^{(1,\infty)}(\mathcal{H}),$ and  $T^{p}\in S_{1}(\mathcal{H}),$ for all  $1< p<\infty$.  In view of 
\eqref{32:ext}, the function
\begin{equation}\label{32:ext:2}
   \textnormal{Det}_{\omega,T}(\lambda):=\lim_{p\rightarrow 1^{+}} \textnormal{Det}_{p,\omega,T}(\lambda)=\exp(\textnormal{\bf{Tr}}_\omega(T)\lambda),
\end{equation}admits analytic extension to the complex plane, which is given by
\begin{equation}
  \textnormal{Det}_{\omega,T}(z):=  \exp(\textnormal{\bf{Tr}}_\omega(T)z),\,z\in \mathbb{C}.
\end{equation}
Define the regularised determinant $\textnormal{\bf Det}_{\omega}$ on the family $I+\mathcal{L}^{(1,\infty)}_{adm}(\mathcal{H})=\{I+T:T\in \mathcal{L}^{(1,\infty)}_{adm}(\mathcal{H})\}$ by
\begin{equation}
    \textnormal{\bf Det}_{\omega}(I+T):= \textnormal{Det}_{\omega,T}(1).
\end{equation}
\end{defn}
\begin{rem}
We can think of the functional  $\textnormal{\bf Det}_{\omega}$ as a regularised determinant for the Dixmier trace $\textnormal{\bf Tr}_{\omega}$.
\end{rem}

\subsection{Connes equivalence theorem}\label{Connessection}  A remarkable result due to Connes computes the Dixmier class for the classical pseudo-differential operators on a closed manifold.

We record that for  a compact orientable manifold without boundary $M$ of dimension $n,$  a  pseudo-differential operator $A$ on $M$ can be defined by using the notion of a local symbol, this means that for any local chart $U$, the operator $A$ has the form $$Au(x)=\int\limits_{T^{*}_xU}e^{2\pi ix\cdot \xi}\sigma^A(x,\xi)\widehat{u}(\xi)\, d\xi.$$ The pseudo-differential operator $A$ is called classical, if $\sigma^A$ admits an asymptotic expansion 
$\sigma^A(x,\xi)\sim \sum_{j=0}^{\infty}\sigma^A_{m-j}(x,\xi)$ in such a way that each function $\sigma_{m-j}(x,\xi)$ is homogeneous in $\xi$ of order $m-j$ for $\xi\neq 0$. The set of classical pseudo-differential operators of order $m$ is denoted by $\Psi^m_{cl}(M)$, and we denote by $\Psi^{m}_{+e}(M)$ the class of positive elliptic classical pseudo-differential operators of order $m\in \mathbb{R}.$
For $A\in\Psi_{cl}^m(M)$, and for $x\in M$, $\int_{|\xi|=1}\sigma_{-n}(x,\xi)\, d\xi$ defines a local density which can be glued over $M$. In this case, the non-commutative residue of $A$ is defined by the expression
\begin{equation}
\textnormal{res}\,(A)=\frac{1}{n(2\pi)^n}\int\limits_{M}\int\limits_{| \xi|=1}\sigma^A_{-n}(x,\xi)\, d\xi\,dx.
\end{equation}
Classical pseudo-differential operators with order $-n,$ where $n=\dim(M)$ are in the Dixmier class $\mathcal{L}^{(1,\infty)}(L^2(M)),$ and moreover:

\begin{thm}[Connes, \cite{Connes94}]\label{equivalence} If $A\in\Psi^{-n}_{cl}(M)$ is a classical pseudo-differential operator, then 
\begin{equation}
  \textnormal{res}(A)=\textnormal{\bf{Tr}}_{\omega}(A),  
\end{equation} for any functional $\omega$ satisfying the conditions in Dixmier's Theorem \ref{Dixm:Theorem}.
\end{thm}

\section{Dixmier traces and Wodzicki residues on compact Lie groups}\label{SectioDixmier}

In this section we investigate the Dixmier traceability of pseudo-differential operators on compact Lie groups. We will apply our analysis to the case of subelliptic pseudo-differential classes $\Psi^{-Q,\mathcal{L}}_{\rho,0}(G\times \widehat{G}).$ In particular, if we replace the sub-Laplacian $\mathcal{L}$ by the Laplace-Beltrami operator our results can be applied to the standard H\"ormander classes $\Psi^{-n}_{\rho,0}(G;\textnormal{loc})=\Psi^{-n}_{\rho,0}(G\times \widehat{G}).$\\

For our further analysis we require the following lemma (see Lemma 8.10 in \cite{CR20}).

\begin{lem}\label{beautifulproof}
For  $0\leqslant\delta< \rho\leqslant 1,$   let us consider a continuous linear operator $P:C^\infty(G)\rightarrow\mathscr{D}'(G)$ with global symbol  $\sigma_P\in {S}^{m,\mathcal{L}}_{\rho,\delta}( G\times \widehat{G})$, of order $m<-Q.$ Then $\textnormal{\bf{Tr}}_{w}(P)=0.$
\end{lem}

\begin{rem}[Dixmier traces of Fourier multipliers]\label{remark:FM} Let $G$ be a compact Lie group and let $A:C^{\infty}(G)\rightarrow C^{\infty}(G)$ be a Fourier multiplier. In terms of the quantisation process in  \cite{Ruz}, there is a  matrix-symbol $\sigma_A:\widehat{G}\rightarrow \cup_{[\xi]\in \widehat{G}}\mathbb{C}^{d_\xi\times d_\xi}$ such that
\begin{equation}\label{vector-valuedq}
    Af(x)=\sum_{[\xi]\in \widehat{G}}d_{\xi}\textnormal{\bf{Tr}}[\xi(x)\sigma_A(\xi)\widehat{f}(\xi)   ],\,\,\,f\in C^{\infty}(G).
\end{equation} It was proved in \cite{cdc20}
 that $A\in \mathcal{L}^{(1,\infty)}_\omega(L^2(G)),$ (see Definition \ref{def:admisible}) if and only if,
 \begin{equation}
    \Vert \sigma_A(\xi)\Vert_{\mathcal{L}^{(1,\infty)}(\widehat{G})}:=\lim_{N\rightarrow\infty}\frac{1}{\log N}\sum_{[\xi]:\langle \xi\rangle\leq N}d_{\xi}\textnormal{\bf{Tr}}(|\sigma_A(\xi)|)<\infty,
\end{equation} and that  $\textnormal{\bf{Tr}}_\omega(A)= \Vert \sigma_A(\xi)\Vert_{\mathcal{L}^{(1,\infty)}(\widehat{G})}.$  In view of Theorem \ref{dixmierdet}, for  $A\in \Psi^{-n}_{+e}(G),$ $n=\dim(G),$ we have
\begin{equation}\label{dixmierdet2}
     \textnormal{Det}_{\omega,A}(\lambda)=\exp\left(\lambda\Vert \sigma_A(\xi)\Vert_{\mathcal{L}^{(1,\infty)}(\widehat{G})}\right),
\end{equation} for $\lambda\in \mathbb{C}$ with  $|\lambda|$ small enough.
In particular, $$\textnormal{Det}_{\omega,A}'(0):=\frac{d}{d\lambda}\textnormal{Det}_{\omega,A}(\lambda)|_{\lambda=0}=\Vert \sigma_A(\xi)\Vert_{\mathcal{L}^{(1,\infty)}(\widehat{G})}.$$
\end{rem}

Now, we will compute the Dixmier trace of positive operators. 
\begin{lem}\label{thmmLieGroup} Let   $A\in \Psi^{-Q,\mathcal{L}}_{\rho,0}(G\times  \widehat{G})$ be a positive  $\mathcal{L}$-elliptic pseudo-differential operator  and let $0<\rho\leq 1$. Assume that the symbol of $A$ admits an asymptotic expansion
\begin{equation*}
    \sigma_A(x,[\xi])\sim \sum_{k=-\infty}^{-Q}\sigma_{k}(x,[\xi]),\,
\end{equation*} in components with decreasing order, which means that, for any $N\in \mathbb{N},$
\begin{equation}\label{asymp:exp}
     \sigma_A(x,[\xi])- \sum_{k=-N-Q}^{-Q}\sigma_{k}(x,[\xi])\in S^{-(N+1)\rho-Q}_{\rho,0}(G\times \widehat{G}),
\end{equation}with the principal symbol $\sigma_{-Q}(x,[\xi])\geq 0,$ being positive.
Then $A\in \mathcal{L}^{(1,\infty)}(L^2(G)),$
\begin{equation}\label{Dix:form:inte}
    \textnormal{\bf{Tr}}_{\omega}(A)=\int\limits_{G}\Vert \sigma_{-Q}(x,[\xi]) \Vert_{\mathcal{L}^{(1,\infty)}(\widehat{G})}dx,
\end{equation}
and we have
\begin{equation}\label{dixmierdet22}
     \textnormal{Det}_{\omega,A}(\lambda)= \exp\left(\textnormal{\bf{Tr}}_{\omega}(A)\lambda\right)=\exp\left(\int\limits_{G}\Vert \sigma_{-Q}(x,[\xi]) \Vert_{\mathcal{L}^{(1,\infty)}(\widehat{G})}dx\cdot \lambda\right),
\end{equation} for $\lambda\in \mathbb{C}$ with  $|\lambda|$ small enough,
and $$\textnormal{Det}_{\omega,A}'(0):=\frac{d}{d\lambda}\textnormal{Det}_{\omega,A}(\lambda)|_{\lambda=0}=\textnormal{\bf{Tr}}_\omega(A).$$
\end{lem}
\begin{proof} For any $x\in G,$ observe that 
 \begin{equation*}
    \Vert \sigma_{-Q}(x,[\xi])\Vert_{\mathcal{L}^{(1,\infty)}(\widehat{G})}:=\lim_{N\rightarrow\infty}\frac{1}{\log N}\sum_{[\xi]:\langle \xi\rangle\leq N}d_{\xi}\textnormal{\bf{Tr}}(|\sigma_{-Q}(x,[\xi])|).
\end{equation*}The estimate $$\textnormal{\bf{Tr}}(|\sigma_{-Q}(x,[\xi])|\widehat{\mathcal{M}}(\xi)^{Q}\widehat{\mathcal{M}}(\xi)^{-Q})\leq\Vert |\sigma_{-Q}(x,[\xi])|\widehat{\mathcal{M}}(\xi)^{Q}\Vert_{\textnormal{op}} \textnormal{\bf{Tr}}(|\sigma_{-Q}(x,[\xi])|)$$ implies that
\begin{align*}
    \Vert \sigma_{-Q}(x,[\xi]) \Vert_{\mathcal{L}^{(1,\infty)}(\widehat{G})} &= \Vert |\sigma_{-Q}(x,[\xi])|\widehat{\mathcal{M}}(\xi)^{Q}\widehat{\mathcal{M}}(\xi)^{-Q} \Vert_{\mathcal{L}^{(1,\infty)}(\widehat{G})}\\
    &\leq\left(\sup_{[\xi]\in \widehat{G}}\Vert\,| \sigma_{-Q}(x,[\xi])|\,\widehat{\mathcal{M}}(\xi)^{Q}\Vert_{\textnormal{op}} \right)\Vert \widehat{\mathcal{M}}(\xi)^{-Q} \Vert_{\mathcal{L}^{(1,\infty)}(\widehat{G})}\\&\lesssim 1.
\end{align*}
Indeed, the functional calculus gives us that $|\sigma_{-Q}(x,[\xi])|\in S^{-Q}_{\rho,0}(G\times \widehat{G}), $ and then,  that $$\sup_{[\xi]\in  \widehat{G}}\Vert\,| \sigma_{-Q}(x,[\xi])|\,\widehat{\mathcal{M}}(\xi)^{Q}\Vert_{\textnormal{op}}\leq C <\infty,$$ uniformly in $x\in G.$ While, the estimate $\Vert \widehat{\mathcal{M}}(\cdot)^{-Q} \Vert_{\mathcal{L}^{(1,\infty)}(\widehat{G})}<\infty$ was proved in Lemma 8.9 of \cite{CR20}. Consequently, the compactness of the group $G,$ implies that  $ \Vert \sigma_{-Q}(x,[\xi]) \Vert_{\mathcal{L}^{(1,\infty)}(\widehat{G})}\in L^1(G).$ So, by the positivity of $A,$ if we prove \eqref{Dix:form:inte},  we deduce that $A\in \mathcal{L}^{(1,\infty)}_\omega(L^2(G))$ in view of Remark \ref{MainRemark}.\\

In view of the asymptotic expansion \eqref{asymp:exp}, we can decompose the operator $A$ into the sum $A=A_0+P,$ where $A_0:=\textnormal{Op}(\sigma_{-Q}),$ and $P$ a pseudo-differential operator in the class $\Psi^{-Q-\rho,\mathcal{L}}_{\rho,\delta}(G\times \widehat{G}).$ Before continuing with the proof, let us remark that in view of the linearity of the functional $\textnormal{\bf{Tr}}_\omega,$ we have
\begin{equation}\label{eq:dix}
    \textnormal{\bf{Tr}}_\omega(A)=\textnormal{\bf{Tr}}_\omega(A_0)+\textnormal{\bf{Tr}}_\omega(P)=\textnormal{\bf{Tr}}_\omega(A_0).
\end{equation}Note that we have used that $\textnormal{\bf{Tr}}_\omega(P)=0$ in view of  Lemma \ref{beautifulproof}. So, \eqref{eq:dix} is the reason of why the symbol $\sigma_{-Q}$ appears in \eqref{dixmierdet22}. Let is compute the Dixmier trace of $A.$

For every $z\in G,$ let us consider the Fourier multiplier $A_z$ associated to the symbol $\sigma(z,\cdot):=\sigma_{-Q}(z,\cdot).$ Observe that any Fourier multiplier $A_z$ is positive because  $\sigma (z,[\xi])\geqslant     0,$ for every $[\xi].$ Also, the  $\mathcal{L}$-ellipticity of $A$ implies the $\mathcal{L}$-ellipticity of $A_{z}$ for every $z\in G.$ Because $A_{z}$ is left-invariant and of subelliptic order $-Q,$ for all $z\in G,$ from Lemma 8.9 of \cite{CR20} we have that $A_{z}\in \mathcal{L}^{(1,\infty)}(L^2(G)),$ and from Theorem 1.1 of \cite{cdc20} (see also Remark  \ref{remark:FM}), we have that the Dixmier trace of any $A_{z}\geq 0,$ is given by
\begin{equation}\label{Dix:Multi:th}
    \textnormal{\bf{Tr}}_\omega(A_z)=\lim_{p\rightarrow 1^{+}}(p-1) \textnormal{\bf{Tr}}(A_z^p)=\Vert \sigma_{-Q}(z,\cdot) \Vert_{\mathcal{L}^{(1,\infty)}(\widehat{G})}.
\end{equation}

In view of the asymptotic expansion \eqref{asymp:exp},  the functional calculus in \cite{CR20} implies that
    \begin{equation}
        A^p=\textnormal{Op}[(x,[\xi])\mapsto\sigma(x,[\xi])^{p}]+R_{p},\quad A_{z}^p=\textnormal{Op}[[\xi]\mapsto\sigma_{-Q}(z,\xi)^{p}],
    \end{equation}where, for any $1<p<\infty,$ $R_{p}$ is a subelliptic pseudo-differential operator of order $-Qp-\rho.$ Moreover, since the order of $A^p$ is $-Qp,$ for all $1<p<\infty,$ in view of Corollary 8.14 of \cite{CR20}, $A^p$ is trace class and $\textnormal{\bf{Tr}}(A^p)=\int\limits_{G}\sum_{[\xi]\in \widehat{G}}d_\xi \textnormal{Tr}[\sigma_{A^p}(x,[\xi])]dx,$ in view of the trace formula in \cite{dr13a:nuclp}.  
    Because $A\geq 0$ we have that
    \begin{align*}
     \textnormal{\bf{Tr}}_\omega(A)&=\lim_{p\rightarrow 1^{+}}(p-1) \textnormal{\bf{Tr}}(A^p)=\lim_{p\rightarrow 1^{+}}(p-1)\int\limits_{G}\sum_{[\xi]\in \widehat{G}}d_\xi \textnormal{Tr}[\sigma_{A^p}(x,[\xi])]dx \\
     &=\lim_{p\rightarrow 1^{+}}(p-1)\int\limits_{G}\sum_{[\xi]\in \widehat{G}}d_\xi \textnormal{Tr}[\sigma_{-Q}(x,[\xi])^p+\sigma_{R_p}(x,[\xi])]dx\\
      &=\lim_{p\rightarrow 1^{+}}(p-1)\left(\int\limits_{G}\sum_{[\xi]\in \widehat{G}}d_\xi \textnormal{Tr}[\sigma_{-Q}(x,[\xi])^p]dx+\int\limits_{G}\sum_{[\xi]\in \widehat{G}}d_\xi \textnormal{Tr}[\sigma_{R_p}(x,[\xi])]dx\right)\\
     &=\lim_{p\rightarrow 1^{+}}(p-1) \left( \textnormal{\bf{Tr}}[\textnormal{Op}(\sigma_{-Q}^{p})] +\textnormal{\bf{Tr}}[R_{p}]\right).
    \end{align*}
  Now, we are going to guarantee the existence of this limit. Firstly, using again Corollary 8.14 of \cite{CR20},  and that the order of $\textnormal{Op}(\sigma_{-Q}^{p})$ is $-Qp,$ for all $1<p<\infty,$ we have that $\textnormal{\bf{Tr}}(\textnormal{Op}(\sigma_{-Q}^{p}))=\int\limits_{G}\sum_{[\xi]\in \widehat{G}}d_\xi \textnormal{Tr}[\sigma_{-Q}^{p}(x,[\xi])]dx,$ in view of the trace formula in \cite{dr13a:nuclp}. Moreover, using  \eqref{Dix:Multi:th} we have that
    \begin{align*}
       \lim_{p\rightarrow 1^{+}}(p-1)  \textnormal{\bf{Tr}}[\textnormal{Op}(\sigma_{-Q}^{p})]
       &=\lim_{p\rightarrow 1^{+}}(p-1)\int\limits_{G}\sum_{[\xi]\in \widehat{G}}d_\xi \textnormal{Tr}[\sigma_{-Q}(x,[\xi])^p]dx\\
        &= \int\limits_{G}\lim_{p\rightarrow 1^{+}}(p-1)\sum_{[\xi]\in \widehat{G}}d_\xi \textnormal{Tr}[\sigma_{-Q}(x,[\xi])^p]dx\\
        &=\int\limits_{G}\lim_{p\rightarrow 1^{+}}(p-1) \textnormal{\bf{Tr}}(A_x^p)dx=\int\limits_{G} \textnormal{\bf{Tr}}_\omega(A_x)dx\\
       &= \int\limits_{G}\Vert \sigma_{-Q}(x,\cdot) \Vert_{\mathcal{L}^{(1,\infty)}(\widehat{G})}dx.
    \end{align*}
We observe that we can interchange the limit with the integral  in the previous argument  in view of the dominated convergence theorem. Indeed, the operator  $\textnormal{Op}(\sigma_{-Q}(x,[\xi])^p)$ of subelliptic order  $m=-Qp,$  has a right-convolution kernel $k_{x,p}$ that  behaves like  (view  \cite[Section 4]{CR20})  $$  
        |k_{x,p}(y)|\lesssim_{\sigma_{A_0}} (1+|y|)^{-\frac{Q+m}{\rho}}=(1+|y|)^{\frac{(p-1)Q}{\rho}}\lesssim 1,$$ for all $p\in (1,2).$ So,  $k_{x,p}$ is locally integrable at the identity $e_G$ for all $p\in (1,2),$ in view of the compactness of $G$ and the continuity of the geodesic norm $y\mapsto |y|.$ Note that the  Fourier inversion formula gives
    \begin{equation}
        k_{x,p}(e_G)=\sum_{[\xi]\in \widehat{G}}d_\xi \textnormal{Tr}[\sigma_{-Q}(x,[\xi])^p].
    \end{equation} So, we have that $|k_{x,p}(e_G)|\leq C,$ with $C>0,$ independent of $p\in(1,2),$  and that $\int_G|k_{x,p}(e_G)|dx\leq C$ for all $1<p\leq 2,$ in view of the compactness of $G.$  The dominated convergence theorem gives 
  \begin{equation}
      \lim_{p\rightarrow 1^{+}}(p-1)\int\limits_{G}k_{x,p}(e_G)dx=\int\limits_{G}\lim_{p\rightarrow 1^{+}}(p-1) k_{x,p}(e_G)dx
  \end{equation}ensuring that
  $$\lim_{p\rightarrow 1^{+}}(p-1)\int\limits_{G}\sum_{[\xi]\in \widehat{G}}d_\xi \textnormal{Tr}[\sigma_{-Q}(x,[\xi])^p]dx
        = \int\limits_{G}\lim_{p\rightarrow 1^{+}}(p-1)\sum_{[\xi]\in \widehat{G}}d_\xi \textnormal{Tr}[\sigma_{-Q}(x,[\xi])^p]dx$$ and the identity
  \begin{equation}
      \lim_{p\rightarrow 1^{+}}(p-1)  \textnormal{\bf{Tr}}[\textnormal{Op}(\sigma_{-Q}^{p})]=\int\limits_{G}\Vert \sigma_{-Q}(x,\cdot) \Vert_{\mathcal{L}^{(1,\infty)}(\widehat{G})}dx.
  \end{equation}  
 We claim that 
$$  \lim_{p\rightarrow 1^{+}}(p-1)\textnormal{\bf{Tr}}(R_p)=0. $$
For the proof, we fix $p\in (1,2),$  let $0<\varepsilon <\frac{\rho}{2}$
 and observe that
\begin{align*}
    |\textnormal{\bf{Tr}}(R_p)|=|\textnormal{\bf{Tr}}(R_p{\mathcal{M}}_{(Q+\varepsilon)p}{\mathcal{M}}_{-(Q+\varepsilon)p})|\leq C\Vert R_p{\mathcal{M}}_{(Q+\varepsilon)p} \Vert_{\mathscr{B}(L^2(G))}|\textnormal{\bf{Tr}}({\mathcal{M}}_{-(Q+\varepsilon)p})|.
\end{align*}Since  
\begin{equation}\label{about:order}
    R_p{\mathcal{M}}_{(Q+\varepsilon)p}\in \Psi^{-\rho+\varepsilon p,\mathcal{L}}_{\rho,0}(G\times \widehat{G}),
\end{equation} 
 the inequalities
$\varepsilon<p\varepsilon<2\varepsilon<\rho,$ and the Calder\'on-vaillancourt theorem (see (iii) in Theorem \ref{calculus}) imply that  $$\forall 1<p<2,\,\,\Vert R_p{\mathcal{M}}_{(Q+\varepsilon)p} \Vert_{\mathscr{B}(L^2(G))}<\infty.$$ 
For any $x\in G,$ the Sobolev embedding theorem implies that 
\begin{align*}
| \textnormal{Op}[(x,[\xi])\mapsto & \sigma_{R_p}(x,[\xi])\widehat{\mathcal{M}}_{(Q+\varepsilon)p}(\xi)] f(x)|\\
\leq     \sup_{ z\in G} |\textnormal{Op}[[\xi]\mapsto &\sigma_{R_p}(z,[\xi])\widehat{\mathcal{M}}_{(Q+\varepsilon)p}(\xi)] f(x)|\\
    &\leq \sum_{|\alpha|\leq {[n/2]+1}}\Vert \textnormal{Op}[[\xi]\mapsto X_{z}^{\alpha}\sigma_{R_p}(z,[\xi])\widehat{\mathcal{M}}_{(Q+\varepsilon)p}(\xi)]f(x)\Vert_{L^2( G_z)}.
\end{align*}
The notation $L^2( G_z)$ indicates that the $L^2$-norm is taken in the $z$-variables. Consequently,
\begin{align*}
    &\Vert R_p{\mathcal{M}}_{(Q+\varepsilon)p} \Vert_{\mathscr{B}(L^2)}\\
    &=\sup_{\Vert f\Vert_{L^2}=1}\Vert \textnormal{Op}[(x,[\xi])\mapsto \sigma_{R_p}(x,[\xi])\widehat{\mathcal{M}}_{(Q+\varepsilon)p}(\xi)] f(x)\Vert_{L^2}\\
    &\leq \sup_{\Vert f\Vert_{L^2}=1} \Vert \sup_{ z\in G} |\textnormal{Op}[[\xi]\mapsto \sigma_{R_p}(z,[\xi])\widehat{\mathcal{M}}_{(Q+\varepsilon)p}(\xi)] f(x)|\Vert_{L^2}\\
    &\leq \sup_{\Vert f\Vert_{L^2}=1}\sum_{|\alpha|\leq {[n/2]+1}}\Vert \textnormal{Op}[[\xi]\mapsto X_{z}^{\alpha}\sigma_{R_p}(z,[\xi])\widehat{\mathcal{M}}_{(Q+\varepsilon)p}(\xi)]f(x)\Vert_{L^2(G_x\times G_z)}\\
    &\leq \sup_{\Vert f\Vert_{L^2}=1}\sum_{|\alpha|\leq {[n/2]+1}}\sup_{z\in G}\Vert \textnormal{Op}[[\xi]\mapsto X_{z}^{\alpha}\sigma_{R_p}(z,[\xi])\widehat{\mathcal{M}}_{(Q+\varepsilon)p}(\xi)]\Vert_{\mathscr{B}(L^2)}\Vert  f\Vert_{L^2}\\
    &= \sum_{|\alpha|\leq {[n/2]+1}}\sup_{z\in G}\sup_{[\xi]\in \widehat{G}}\Vert X_{z}^{\alpha}\sigma_{R_p}(z,[\xi])\widehat{\mathcal{M}}_{(Q+\varepsilon)p}(\xi)\Vert_{\textnormal{op}}\\
    &\lesssim_{n}\sup_{|\alpha|\leq {[n/2]+1}} \sup_{ (z,[\xi])\in G\times \widehat{G}}\Vert X_{z}^{\alpha}\sigma_{R_p}(z,[\xi])\widehat{\mathcal{M}}_{(Q+\varepsilon)p}(\xi)\widehat{\mathcal{M}}^{\rho-\varepsilon p}(\xi)\Vert_{\textnormal{op}} \Vert\widehat{\mathcal{M}}^{-\rho+\varepsilon p}(\xi)\Vert_{\textnormal{op}}\\
    &\lesssim 1, 
\end{align*}where we have used that $$ \sup_{ (z,[\xi])\in G\times \widehat{G}}\Vert X_{z}^{\alpha} \sigma_{R_p}(z,[\xi])\widehat{\mathcal{M}}_{(Q+\varepsilon)p}(\xi)\widehat{\mathcal{M}}^{\rho-\varepsilon p}(\xi)\Vert_{\textnormal{op}}=O(1),\textnormal{ } \Vert\widehat{\mathcal{M}}^{-\rho+\varepsilon p}(\xi)\Vert_{\textnormal{op}}\leq 1,$$ in view of \eqref{about:order} and the fact that $-\rho+\varepsilon p<0$. The analysis above, and the positivity of the operator ${\mathcal{M}}_{-(Q+\varepsilon)p},$  imply that
\begin{align*}
  |\lim_{p\rightarrow 1^{+}}(p-1)\textnormal{\bf{Tr}}(R_p)| &=   \lim_{p\rightarrow 1^{+}}(p-1) |\textnormal{\bf{Tr}}(R_p)|\\
  &\leq \lim_{p\rightarrow 1^{+}}(p-1)C\Vert R_p{\mathcal{M}}_{(Q+\varepsilon)p} \Vert_{\mathscr{B}(L^2(G))}|\textnormal{\bf{Tr}}({\mathcal{M}}_{-(Q+\varepsilon)p})|\\
  &\lesssim \lim_{p\rightarrow 1^{+}}(p-1)|\textnormal{\bf{Tr}}({\mathcal{M}}_{-(Q+\varepsilon)p})|\\
  & =\lim_{p\rightarrow 1^{+}}(p-1)\textnormal{\bf{Tr}}({\mathcal{M}}^{-(Q+\varepsilon)p})\\
  &=\textnormal{\bf{Tr}}_\omega({\mathcal{M}}^{-(Q+\varepsilon)}).
\end{align*}
Because of Lemma \ref{beautifulproof}, we have that $\textnormal{\bf{Tr}}_\omega({\mathcal{M}}^{-(Q+\varepsilon)})=0,$ and then   $$\lim_{p\rightarrow 1^{+}}(p-1)\textnormal{\bf{Tr}}(R_p)=0.$$  Consequently,
\begin{align*}
     \textnormal{\bf{Tr}}_\omega(A)=\lim_{p\rightarrow 1^{+}}(p-1) \left( \textnormal{\bf{Tr}}[\textnormal{Op}(\sigma_{-Q}^{p})] +\textnormal{\bf{Tr}}[R_{p}]\right)=\int\limits_{G}\Vert \sigma_{-Q}(x,\cdot) \Vert_{\mathcal{L}^{(1,\infty)}(\widehat{G})}dx.
\end{align*}Thus, we end the proof by applying Lemma \ref{lemmadix}. 
\end{proof}
Now, we will use the functional calculus to remove the positivity assumptions in Lemma \ref{thmmLieGroup}.

\begin{thm}\label{thmmLieGroup:2} Let   $A\in \Psi^{-Q,\mathcal{L}}_{\rho,0}(G\times  \widehat{G})$ be an $\mathcal{L}$-elliptic pseudo-differential operator  and let $0<\rho\leq 1$. Assume that the symbol of $A$ admits an asymptotic expansion
\begin{equation*}
    \sigma_A(x,[\xi])\sim \sum_{k=-\infty}^{-Q}\sigma_{k}(x,[\xi]),\,
\end{equation*}in components with decreasing order, which means that, for any $N\in \mathbb{N},$
\begin{equation*}
     \sigma_A(x,[\xi])- \sum_{k=-N-Q}^{-Q}\sigma_{k}(x,[\xi])\in S^{-(N+1)\rho-Q}_{\rho,0}(G\times \widehat{G}).
\end{equation*}
Then $|A|\in \mathcal{L}^{(1,\infty)}(L^2(G)),$
\begin{equation*}
    \textnormal{\bf{Tr}}_{\omega}(|A|)=\int\limits_{G}\Vert \sigma_{-Q}(x,[\xi]) \Vert_{\mathcal{L}^{(1,\infty)}(\widehat{G})}dx,
\end{equation*}
and we have
\begin{equation}
     \textnormal{Det}_{\omega,|A|}(\lambda)= \exp\left(\textnormal{\bf{Tr}}_{\omega}(|A|)\lambda\right)=\exp\left(\int\limits_{G}\Vert \sigma_{-Q}(x,[\xi]) \Vert_{\mathcal{L}^{(1,\infty)}(\widehat{G})}dx\cdot \lambda\right),
\end{equation} for $\lambda\in \mathbb{C}$ with  $|\lambda|$ small enough,
and $$\textnormal{Det}_{\omega,|A|}'(0):=\frac{d}{d\lambda}\textnormal{Det}_{\omega,|A|}(\lambda)|_{\lambda=0}=\textnormal{\bf{Tr}}_\omega(|A|).$$
\end{thm}
\begin{proof} From the subelliptic calculus (see (i) in  Theorem \ref{calculus}), we have that the matrix-valued symbol $\sigma_{A^*A}$ of $A^*A,$ admits the asymptotic expansion 
\begin{equation*}
    \sigma_{A^*A}(x,[\xi])\sim \sum_{k=-\infty}^{-2Q}\tilde{\sigma}_{k}(x,[\xi]),\,
\end{equation*}in components with decreasing order, which means that, for any $N\in \mathbb{N},$
\begin{equation*}
     \sigma_A(x,[\xi])- \sum_{k=-N-2Q}^{-2Q}{\tilde\sigma}_{k}(x,[\xi])\in S^{-(N+1)\rho-2Q}_{\rho,0}(G\times \widehat{G}),
\end{equation*}with the principal term given by
\begin{equation*}
   {\tilde\sigma}_{-2Q}(x,[\xi])={\sigma}_{-Q}(x,[\xi])^*{\sigma}_{-Q}(x,[\xi])=|{\sigma}_{-Q}(x,[\xi])|^2, 
\end{equation*}and consequently positive for any $(x,[\xi]).$ The functional calculus in \cite{CR20}, implies that the symbol of $|A|:=\sqrt{A^*A},$ admits an asymptotic expansion of the type
\begin{equation*}
    \sigma_{|A|}(x,[\xi])\sim \sum_{k=-\infty}^{-Q}{\sigma'}_{k}(x,[\xi]),\,
\end{equation*}that is, for any $N\in \mathbb{N},$
\begin{equation*}
     \sigma_{|A|}(x,[\xi])- \sum_{k=-N-Q}^{-Q}{\sigma'}_{k}(x,[\xi])\in S^{-(N+1)\rho-Q}_{\rho,0}(G\times \widehat{G}),
\end{equation*}with the principal term given by
\begin{equation*}
   {\sigma'}_{-Q}(x,[\xi])=\sqrt{{\sigma}_{-Q}(x,[\xi])^*{\sigma}_{-Q}(x,[\xi])}=|{\sigma}_{-Q}(x,[\xi])|, 
\end{equation*}which is then  positive for any $(x,[\xi]).$ Applying Lemma \ref{thmmLieGroup} to $|A|,$ and observing that 
\begin{equation*}
     \Vert \sigma'_{-Q}(x,[\xi]) \Vert_{\mathcal{L}^{(1,\infty)}(\widehat{G})}   =\Vert \sigma_{-Q}(x,[\xi]) \Vert_{\mathcal{L}^{(1,\infty)}(\widehat{G})},
\end{equation*}we conclude that
\begin{equation*}
    \textnormal{\bf{Tr}}_{\omega}(|A|)=\int\limits_{G}\Vert \sigma_{-Q}(x,[\xi]) \Vert_{\mathcal{L}^{(1,\infty)}(\widehat{G})}dx.
\end{equation*} We end the proof by applying Lemma \ref{lemmadix}.
\end{proof}
In the case of a positive subelliptic operator we can compute its Dixmier trace as follows.

\begin{cor}\label{thmmLieGroup:2:2} Let   $A\in \Psi^{-Q,\mathcal{L}}_{\rho,0}(G\times  \widehat{G})$ be a positive $\mathcal{L}$-elliptic pseudo-differential operator  and let $0<\rho\leq 1$. Assume that the symbol of $A$ admits an asymptotic expansion
\begin{equation*}
    \sigma_A(x,[\xi])\sim \sum_{k=-\infty}^{-Q}\sigma_{k}(x,[\xi]),\,
\end{equation*}in components with decreasing order, which means that, for any $N\in \mathbb{N},$
\begin{equation}\label{asymp:exp:2:2}
     \sigma_A(x,[\xi])- \sum_{k=-N-Q}^{-Q}\sigma_{k}(x,[\xi])\in S^{-(N+1)\rho-Q}_{\rho,0}(G\times \widehat{G}).
\end{equation}
Then $A\in \mathcal{L}^{(1,\infty)}(L^2(G))$ and its Dixmier trace can be computed as
\begin{equation}
    \textnormal{\bf{Tr}}_{\omega}(A)=\int\limits_{G}\Vert \sigma_{-Q}(x,[\xi]) \Vert_{\mathcal{L}^{(1,\infty)}(\widehat{G})}dx.
\end{equation}
Furthermore, we have
\begin{equation}\label{dixmierdet22:2:2}
     \textnormal{Det}_{\omega,A}(\lambda)= \exp\left(\textnormal{\bf{Tr}}_{\omega}(A)\lambda\right)=\exp\left(\int\limits_{G}\Vert \sigma_{-Q}(x,[\xi]) \Vert_{\mathcal{L}^{(1,\infty)}(\widehat{G})}dx\cdot \lambda\right),
\end{equation} for $\lambda\in \mathbb{C}$ with  $|\lambda|$ small enough,
$$\textnormal{Det}_{\omega,A}'(0):=\frac{d}{d\lambda}\textnormal{Det}_{\omega,A}(\lambda)|_{\lambda=0}=\textnormal{\bf{Tr}}_\omega(A),$$
and the determinant formula
 $$\textnormal{\bf Det}_{\omega}(I+A)=\exp\left(\int\limits_{G}\Vert \sigma_{-Q}(x,[\xi]) \Vert_{\mathcal{L}^{(1,\infty)}(\widehat{G})}dx\right)$$ holds true.
\end{cor}
\begin{proof}
Since $A$ is positive, we can apply Theorem \ref{thmmLieGroup:2} to $A=|A|$ and the proof can be concluded.
\end{proof}
\begin{cor}\label{Cor:Lidskii} Under the hypothesis of Lemma \ref{thmmLieGroup}, we have
\begin{equation}\label{formulaLidskii}
    \textnormal{\bf{Tr}}_{\omega}(A)=\int\limits_{G}\Vert \sigma_{-Q}(x,[\xi]) \Vert_{\mathcal{L}^{(1,\infty)}(\widehat{G})}dx=\lim_{N\rightarrow \infty} \frac{1}{\log(N)}\sum_{j=1}^{N}s_{j}(A),
\end{equation}where $s_{j}(A)$ denotes the sequence of singular values of $A$ with multiplicities taken into account.
\end{cor}
\begin{proof}
It follows from the Dixmier Theorem \ref{Dixm:Theorem} that $ \textnormal{\bf Tr}_\omega(A)= \textnormal{\bf Tr}_{\textnormal{Dix}}(A),$ provided that the limit $\textnormal{\bf Tr}_{\textnormal{Dix}}(A)$ exists. However, the existence of the limit $\lim_{p\rightarrow 1^{+}}(p-1)\textnormal{\bf Tr}(A^p),$ is equivalent to the existence of $\textnormal{\bf Tr}_{\textnormal{Dix}}(A),$ \cite[Page 313]{Connes94} and in the case of their existence they have the same value, that is $\textnormal{\bf Tr}_\omega(A)$.   The identity in \eqref{formulaLidskii} now follows from the proof of   Lemma \ref{thmmLieGroup}, where we have shown that $\textnormal{\bf Tr}_\omega(A)=\lim_{p\rightarrow 1^{+}}(p-1)\textnormal{\bf Tr}(A^p).$
\end{proof}

If in Corollary \ref{thmmLieGroup:2:2} we replace the sub-Laplacian $\mathcal{L}$ by the Laplace-Beltrami operator on $G,$ then $Q=n=\dim(G)$ and we obtain the following result for the standard  H\"ormander classes on $G,$ in view of the equivalence of classes  $\Psi^{-n}_{\rho,0}(G\times \widehat{G})=\Psi^{-n}_{\rho,0}(G;\textnormal{loc})$ for all $0<\rho\leq 1.$ 

\begin{cor}\label{WRGroup:3} Let   $A\in \Psi^{-n}_{\rho,0}(G, \textnormal{loc})$ with $0<\rho\leq 1,$ be a positive elliptic  pseudo-differential  operator on $G$. Assume that the symbol of $A$ admits an asymptotic expansion
\begin{equation}
    \sigma_A(x,[\xi])\sim \sum_{k=-\infty}^{-n}\sigma_{k}(x,[\xi]),\,
\end{equation}in components with decreasing order, which means that, for any $N\in \mathbb{N},$
\begin{equation}\label{asymp:exp:3}
     \sigma_A(x,[\xi])- \sum_{k=-N-n}^{-n}\sigma_{k}(x,[\xi])\in S^{-(N+1)\rho-n}_{\rho,0}(G\times \widehat{G}).
\end{equation}
Then $A\in \mathcal{L}^{(1,\infty)}(L^2(G))$ and 
\begin{equation}
    \textnormal{\bf Tr}_\omega(A)=\int\limits_{G}\Vert \sigma_{-n}(g,[\xi]) \Vert_{\mathcal{L}^{(1,\infty)}(\widehat{G})}dg.
\end{equation}
\end{cor}

In view of the Connes equivalence Theorem \ref{equivalence}  and the equivalence of classes $\Psi^{-n}_{1,0}(G\times \widehat{G})=\Psi^{-n}_{1,0}(G;\textnormal{loc})$ proved in \cite{RuzhanskyTurunenWirth2014} we have:
\begin{cor}\label{WRGroup} Let   $A\in \Psi^{-n}_{cl}(G, \textnormal{loc})$ be a positive and  elliptic  pseudo-differential  operator on $G$. Assume that the symbol of $A$ admits an asymptotic expansion
\begin{equation}
    \sigma_A(x,[\xi])\sim \sum_{k=-\infty}^{-n}\sigma_{k}(x,[\xi]),\,
\end{equation}in components with decreasing order, which means that, for any $N\in \mathbb{N},$
\begin{equation}\label{asymp:exp:4}
     \sigma_A(x,[\xi])- \sum_{k=-N-n}^{-n}\sigma_{k}(x,[\xi])\in S^{-(N+1)-n}_{1,0}(G\times \widehat{G}).
\end{equation}
Then the Wodzicki residue of $A$ can be computed according to the formula
\begin{equation}
    \textnormal{res}(A)=\int\limits_{G}\Vert \sigma_{-n}(x,[\xi]) \Vert_{\mathcal{L}^{(1,\infty)}(\widehat{G})}dx.
\end{equation}
\end{cor}

Now, we present the main theorem of this section.

\begin{thm}\label{Residue:Thm} Let   $A\in \Psi^{-Q,\mathcal{L}}_{\rho,0}(G\times  \widehat{G})$ be an $\mathcal{L}$-elliptic pseudo-differential operator  and let $0<\rho\leq 1$. Assume that the symbol of $A$ admits an asymptotic expansion
\begin{equation*}
    \sigma_A(x,[\xi])\sim \sum_{k=-\infty}^{-Q}\sigma_{k}(x,[\xi]),\,
\end{equation*}in components with decreasing order, which means that, for any $N\in \mathbb{N},$
\begin{equation*}
     \sigma_A(x,[\xi])- \sum_{k=-N-Q}^{-Q}\sigma_{k}(x,[\xi])\in S^{-(N+1)\rho-Q}_{\rho,0}(G\times \widehat{G}).
\end{equation*}
Then $A\in \mathcal{L}^{(1,\infty)}(L^2(G)),$ and its Dixmier trace is given by
\begin{eqnarray*}
     \textnormal{\bf{Tr}}_{\omega}(A)=\int\limits_{G}\left(\Vert \textnormal{Re}(\sigma_{-Q}(x,[\xi]))^{+} \Vert_{\mathcal{L}^{(1,\infty)}(\widehat{G})}-\Vert \textnormal{Re}(\sigma_{-Q}(x,[\xi]))^{-} \Vert_{\mathcal{L}^{(1,\infty)}(\widehat{G})}
     \right)dx\\
    +i\int\limits_{G}\left(\Vert \textnormal{Im}(\sigma_{-Q}(x,[\xi]))^{+} \Vert_{\mathcal{L}^{(1,\infty)}(\widehat{G})}-\Vert \textnormal{Im}(\sigma_{-Q}(x,[\xi]))^{-} \Vert_{\mathcal{L}^{(1,\infty)}(\widehat{G})}
    \right)dx.
\end{eqnarray*}
\end{thm}
\begin{proof}In view of the polar decomposition $A=|A|U$ of $A,$ where $U$ is a unitary operator on $L^2(G)$ (see e.g. \cite{goh:trace}),  that $\mathcal{L}^{(1,\infty)}(L^2(G))$ is an ideal of operators in the $C^*$-algebra $\mathscr{B}(L^2(G)),$ that $|A|\in \mathcal{L}^{(1,\infty)}(L^2(G)) $ in view of Theorem \ref{thmmLieGroup:2}, we deduce that $A\in \mathcal{L}^{(1,\infty)}(L^2(G)).$ Now we will use the linearity properties of the Dixmier trace in order to compute the Dixmier trace of $A.$

We use the decomposition of $A$ into its real and imaginary part as in Remark \ref{Deco:remark},
\begin{equation*}
    \textnormal{Re}(A):=\frac{A+A^*}{2},\,\, \textnormal{Im}(A):=\frac{A-A^*}{2i},
\end{equation*}and the decomposition of $\textnormal{Re}(A)$ and $\textnormal{Im}(A)$ into their positive and negative parts,
\begin{eqnarray*}
  \textnormal{Re}(A)^{+}:=\frac{\textnormal{Re}(A)+|\textnormal{Re}(A)|}{2},\,\, \textnormal{Re}(A)^{-}:=\frac{|\textnormal{Re}(A)|-\textnormal{Re}(A)}{2},
\end{eqnarray*}
and 
\begin{eqnarray*}
  \textnormal{Im}(A)^{+}:=\frac{\textnormal{Im}(A)+|\textnormal{Im}(A)|}{2},\,\, \textnormal{Im}(A)^{-}:=\frac{|\textnormal{Im}(A)|-\textnormal{Im}(A)}{2}.
\end{eqnarray*}Now, the operator $A$ can be written as
\begin{align*}
    A&= \textnormal{Re}(A)+i\textnormal{Im}(A)\\
    &=\left(\textnormal{Re}(A)^{+}-\textnormal{Re}(A)^{-}\right)+i\left(\textnormal{Im}(A)^{+}-\textnormal{Im}(A)^{-}\right).
\end{align*}So, by the linearity of the Dixmier trace $\textnormal{Tr}_\omega$ we have
\begin{align*}
    \textnormal{Tr}_\omega (A)&= \textnormal{Tr}_\omega(\textnormal{Re}(A))+i\textnormal{Tr}_\omega(\textnormal{Im}(A))\\
    &=\left(\textnormal{Tr}_\omega(\textnormal{Re}(A)^{+})-\textnormal{Tr}_\omega(\textnormal{Re}(A)^{-})\right)+i\left(\textnormal{Tr}_\omega(\textnormal{Im}(A)^{+})-\textnormal{Tr}_\omega(\textnormal{Im}(A)^{-})\right).
\end{align*}Now, we will exploit the subelliptic functional calculus in \cite{CR20} in order to compute the symbols of the positive operators $\textnormal{Re}(A)^{+},\textnormal{Re}(A)^{-},\textnormal{Im}(A)^{+},\textnormal{Im}(A)^{-}.$ Indeed, we have
\begin{align*}
    \sigma_{\textnormal{Re}(A)}(x,[\xi])= \sigma_{\frac{A+A^*}{2}}(x,[\xi])
    =\textnormal{Re}(\sigma_{-Q}(x,[\xi]))+\textnormal{lower order terms},
\end{align*}
\begin{align*}
    \sigma_{\textnormal{Im}(A)}(x,[\xi])= \sigma_{\frac{A-A^*}{2i}}(x,[\xi])
    =\textnormal{Im}(\sigma_{-Q}(x,[\xi]))+\textnormal{lower order terms},
\end{align*}
\begin{align*}
    \sigma_{\textnormal{Re}(A)^{+}}(x,[\xi])= \sigma_{\frac{\textnormal{Re}(A)+|\textnormal{Re}(A)|}{2}}(x,[\xi])
    =\textnormal{Re}(\sigma_{-Q}(x,[\xi]))^{+}+\textnormal{lower order terms},
\end{align*}
\begin{align*}
    \sigma_{\textnormal{Re}(A)^{-}}(x,[\xi])= \sigma_{\frac{|\textnormal{Re}(A)|-\textnormal{Re}(A)}{2}}(x,[\xi])
    =\textnormal{Re}(\sigma_{-Q}(x,[\xi]))^{-}+\textnormal{lower order terms},
\end{align*}and 
\begin{align*}
    \sigma_{\textnormal{Im}(A)^{+}}(x,[\xi])= \sigma_{\frac{\textnormal{Im}(A)+|\textnormal{Im}(A)|}{2}}(x,[\xi])
    =\textnormal{Im}(\sigma_{-Q}(x,[\xi]))^{+}+\textnormal{lower order terms},
\end{align*}
\begin{align*}
    \sigma_{\textnormal{Im}(A)^{-}}(x,[\xi])= \sigma_{\frac{|\textnormal{Im}(A)|-\textnormal{Im}(A)}{2}}(x,[\xi])
    =\textnormal{Im}(\sigma_{-Q}(x,[\xi]))^{-}+\textnormal{lower order terms}.
\end{align*}
Now, by applying Corollary \ref{thmmLieGroup:2:2}, we can eliminate the lower terms when computing the Dixmier trace, and  we have the following  formulae

     \begin{equation*}
         \textnormal{\bf{Tr}}_{\omega}(\textnormal{Re}(A)^{+})=\int\limits_{G}\Vert \textnormal{Re}(\sigma_{-Q}(x,[\xi]))^{+} \Vert_{\mathcal{L}^{(1,\infty)}(\widehat{G})}dx
    \end{equation*}
    \begin{equation*}
         \textnormal{\bf{Tr}}_{\omega}(\textnormal{Re}(A)^{-})=\int\limits_{G}\Vert \textnormal{Re}(\sigma_{-Q}(x,[\xi]))^{-} \Vert_{\mathcal{L}^{(1,\infty)}(\widehat{G})}dx
    \end{equation*}
    \begin{equation*}
         \textnormal{\bf{Tr}}_{\omega}(\textnormal{Im}(A)^{+})=\int\limits_{G}\Vert \textnormal{Im}(\sigma_{-Q}(x,[\xi]))^{+} \Vert_{\mathcal{L}^{(1,\infty)}(\widehat{G})}dx
    \end{equation*}
    \begin{equation*}
         \textnormal{\bf{Tr}}_{\omega}(\textnormal{Im}(A)^{-})=\int\limits_{G}\Vert \textnormal{Im}(\sigma_{-Q}(x,[\xi]))^{-} \Vert_{\mathcal{L}^{(1,\infty)}(\widehat{G})}dx.
    \end{equation*}In view of the linearity of the Dixmier trace $ \textnormal{\bf{Tr}}_{\omega}$ and the previous trace formula we end the proof.
\end{proof}

Now, let us replace the  sub-Laplacian $\mathcal{L}$ by the Laplace-Beltrami operator on $G$ in Theorem  \ref{Residue:Thm}. As mentioned above, we obtain $Q=n=\dim(G)$ and we obtain the following result for the standard  H\"ormander classes on $G,$ in view of the equivalence of classes  $\Psi^{-n}_{\rho,0}(G\times \widehat{G})=\Psi^{-n}_{\rho,0}(G;\textnormal{loc})$ for all $0<\rho\leq 1,$ see \cite{RuzhanskyTurunenWirth2014}. 

\begin{cor}\label{WRGroup:3:3} Let   $A\in \Psi^{-n}_{\rho,0}(G, \textnormal{loc})$ with $0<\rho\leq 1,$ be an elliptic  pseudo-differential  operator on $G$. Assume that the symbol of $A$ admits an asymptotic expansion
\begin{equation}
    \sigma_A(x,[\xi])\sim \sum_{k=-\infty}^{-n}\sigma_{k}(x,[\xi]),\,
\end{equation}in components with decreasing order, which means that, for any $N\in \mathbb{N},$
\begin{equation}
     \sigma_A(x,[\xi])- \sum_{k=-N-n}^{-n}\sigma_{k}(x,[\xi])\in S^{-(N+1)\rho-n}_{\rho,0}(G\times \widehat{G}).
\end{equation}
Then $A\in \mathcal{L}^{(1,\infty)}(L^2(G))$ and its Dixmier trace is given by
\begin{eqnarray*}
     \textnormal{\bf{Tr}}_{\omega}(A)=\int\limits_{G}\left(\Vert \textnormal{Re}(\sigma_{-Q}(x,[\xi]))^{+} \Vert_{\mathcal{L}^{(1,\infty)}(\widehat{G})}-\Vert \textnormal{Re}(\sigma_{-Q}(x,[\xi]))^{-} \Vert_{\mathcal{L}^{(1,\infty)}(\widehat{G})}
     \right)dx\\
    +i\int\limits_{G}\left(\Vert \textnormal{Im}(\sigma_{-Q}(x,[\xi]))^{+} \Vert_{\mathcal{L}^{(1,\infty)}(\widehat{G})}-\Vert \textnormal{Im}(\sigma_{-Q}(x,[\xi]))^{-} \Vert_{\mathcal{L}^{(1,\infty)}(\widehat{G})}
    \right)dx.
\end{eqnarray*}
\end{cor}

In view of the Connes equivalence Theorem \ref{equivalence}  and the equivalence of classes $\Psi^{-n}_{1,0}(G\times \widehat{G})=\Psi^{-n}_{1,0}(G;\textnormal{loc})$ we have:
\begin{cor}\label{WRGroup:2} Let   $A\in \Psi^{-n}_{cl}(G, \textnormal{loc})$ be an  elliptic  pseudo-differential  operator on $G$. Assume that the symbol of $A$ admits an asymptotic expansion
\begin{equation}
    \sigma_A(x,[\xi])\sim \sum_{k=-\infty}^{-n}\sigma_{k}(x,[\xi]),\,
\end{equation}in components with decreasing order, which means that, for any $N\in \mathbb{N},$
\begin{equation}
     \sigma_A(x,[\xi])- \sum_{k=-N-n}^{-n}\sigma_{k}(x,[\xi])\in S^{-(N+1)-n}_{1,0}(G\times \widehat{G}).
\end{equation}
Then the Wodzicki residue of $A$ can be computed according to the formula
\begin{eqnarray*}
     \textnormal{res}(A)=\int\limits_{G}\left(\Vert \textnormal{Re}(\sigma_{-Q}(x,[\xi]))^{+} \Vert_{\mathcal{L}^{(1,\infty)}(\widehat{G})}-\Vert \textnormal{Re}(\sigma_{-Q}(x,[\xi]))^{-} \Vert_{\mathcal{L}^{(1,\infty)}(\widehat{G})}
     \right)dx\\
    +i\int\limits_{G}\left(\Vert \textnormal{Im}(\sigma_{-Q}(x,[\xi]))^{+} \Vert_{\mathcal{L}^{(1,\infty)}(\widehat{G})}-\Vert \textnormal{Im}(\sigma_{-Q}(x,[\xi]))^{-} \Vert_{\mathcal{L}^{(1,\infty)}(\widehat{G})}
    \right)dx.
\end{eqnarray*}
\end{cor}

\section{Final remarks on  compact manifolds}\label{fremarks}
We now illustrate some connections between the determinant of operators in the setting of the  invariance  notion recalled in Section \ref{sec3} and  introduced in \cite{dr14a:fsymbsch},   and the problem of computing the  Dixmier trace of elliptic operators. We denote by $\Psi^{m}_{+e}(M)$ the class of positive elliptic classical pseudo-differential operators of order $m\in \mathbb{R}.$\\

If $A\in \Psi^{-{n}}_{+e}(M)$, then $A $ is not necessarily an operator of trace class. Indeed, the best information that we know about the Schwartz kernel $K_A$ of $A,$ is its logarithmic behaviour over the diagonal (see H\"ormander \cite{Hormander1985III}).  So, the determinant function $\textnormal{Det}(I+\lambda \cdot )$ is not a well defined functional on the class of operators $\Psi^{-{n}}_{+e}(M)$.  However, when studying the  regularised version 
\begin{equation}
    \textnormal{Det}_{\omega,A}(\lambda):=\lim_{p\rightarrow 1^{+}} \textnormal{Det}(I+\lambda A^p)^{p-1},
\end{equation} of $`\textnormal{Det}(I+\lambda A )$', we can observe that it  can be computed in terms of the global symbol $\sigma_{A,E}\equiv (\sigma_{A,E}(\ell))_{\ell\in \mathbb{N}}$ of $A,$ and of its Dixmier trace (and also, in view of the Connes equivalence in Theorem \ref{equivalence}, of its non-commutative residue).

\begin{rem}\label{remarkdix}
Let  $M$ be  a compact manifold of dimension $n$ without boundary endowed with a volume form $dx$, and let $E$ be  a positive elliptic classical pseudo-differential operator of order $\nu>0.$ Let us denote the eigenvalues of $E$ by $\{\lambda_j\}_{j\geq 0}.$
If $A:C^{\infty}(M)\rightarrow C^{\infty}(M) $ is invariant with respect to the subspaces $H_{j}=\textnormal{Ker}(E-\lambda_jI),$ in terms of the symbol $\sigma_{A,E}(\ell)$ of $A,$ let us define the functional  
\begin{equation}\label{eq1}
\Vert \sigma_{A,E}\Vert_{\mathcal{L}^{(1,\infty)}(\mathbb{N},E)}:=\frac{1}{n}\lim_{N\rightarrow\infty}\frac{ 1  }{\log N}\sum_{ l:(1+\lambda_{l})^{\frac{1}{\nu}}\leq N}\textnormal{\bf{Tr}}(|\sigma_{A,E}(l)|).
\end{equation}It was proved in \cite{cdc20}, that $A\in \mathcal{L}^{(1,\infty)}(L^2(M)) ,$ if and only if, $$\Vert \sigma_{A,E}\Vert_{\mathcal{L}^{(1,\infty)}(\mathbb{N},E)}<\infty.$$ In particular, if $A$ is positive (with respect to the inner product on $L^2(M)$), $$\textnormal{\bf{Tr}}_{\omega}(A)=\Vert \sigma_{A,E}\Vert_{\mathcal{L}^{(1,\infty)}(\mathbb{N},E)}.$$
\end{rem} In view of Lemma \ref{lemmadix} and Remark \ref{remarkdix} we have the following theorem.
\begin{thm}\label{thmmanifold} Let   $A\in \Psi^{-n}_{+e}(M)$ be a positive and elliptic operator. Then $A\in \mathcal{L}^{(1,\infty)}(L^2(M)),$ and 
\begin{equation}\label{dixmierdet}
     \textnormal{Det}_{\omega,A}(\lambda)= \exp\left(\textnormal{\bf{Tr}}_{\omega}(A)\lambda\right)=\exp\left(\lambda \Vert \sigma_{A,E}\Vert_{\mathcal{L}^{(1,\infty)}(\mathbb{N},E)}\right),
\end{equation} for $\lambda\in \mathbb{C}$ with  $|\lambda|$ small enough and any  $E\in \Psi^{-n}_{+e}(M)$ commuting with $A$.
In particular, $$\textnormal{Det}_{\omega,A}'(0):=\frac{d}{d\lambda}\textnormal{Det}_{\omega,A}(\lambda)|_{\lambda=0}=\textnormal{\bf{Tr}}_\omega(A)=\Vert \sigma_{A,E}\Vert_{\mathcal{L}^{(1,\infty)}(\mathbb{N},E)},$$
and 
\begin{equation}
    \textnormal{\bf{Det}}_\omega(I+A)=\exp(\Vert \sigma_{A,E}\Vert_{\mathcal{L}^{(1,\infty)}(\mathbb{N},E)}).
\end{equation}
\end{thm}
\begin{proof}
Let us prove that $A$ satisfies the hypothesis in Lemma \ref{lemmadix}. Indeed, the fact that $A\in \Psi^{-n}_{+}(M),$ implies that $A\in \mathcal{L}^{(1,\infty)}(L^2(M))$ (see  \cite{Connes94}). The functional calculus for elliptic operators implies that $A^{p}\in \Psi^{-np}_{+}(M),$  and consequently, for $1<p<\infty ,$ $\textnormal{\bf{Tr}}_{\omega}(A^{p})=0.$ Indeed, the condition $A^{p}\in \Psi^{-np}_{+}(M),$ implies that $A^p\in S_{1}(L^2(M))$ which implies that its Dixmier functional vanishes  (see  \cite{Connes94}). In view of Lemma  \ref{lemmadix} we have proved the left hand side of \eqref{dixmierdet}. The right-hand side of  \eqref{dixmierdet} is a consequence of Remark \ref{remarkdix}. 
\end{proof}

\end{document}